\documentclass{amsart}

\usepackage{amssymb}
\usepackage{amsmath}
\usepackage{amsthm}
\usepackage[mathscr]{eucal}
\usepackage{mathpazo}

\newtheorem{thrm}{Theorem}[section]
\newtheorem{lemma}[thrm]{Lemma}
\newtheorem{prop}[thrm]{Proposition}

\newtheorem*{main*}{Main Theorem}

\theoremstyle{definition}
\newtheorem{defn}[thrm]{Definition}

\theoremstyle{remark}

\newtheorem*{question*}{Question}

\numberwithin{equation}{section}

\newcommand{\dbar}{$\bar{\partial}$}
\newcommand{\mdbar}{\bar{\partial}}

\newcommand{\zb}{\bar{z}}

\newcommand{\Lb}{\overline{L}}

\newcommand{\omegab}{\bar{\omega}}
\newcommand{\gammab}{\bar{\gamma}}

\newcommand{\lrl}{\mathcal{L}}
\newcommand{\lrs}{\mathcal{S}}

\usepackage{mathrsfs,latexsym,url,setspace}
\onehalfspace

\begin{document}

\bibliographystyle{plain}

\title{
Dirichlet to Neumann operators and the 
 \dbar-Neumann problem}

\author{Dariush Ehsani}

\address{
Hochschule Merseburg\\
Eberhard-Leibnitz-Str. 2\\
 D-06217 Merseburg\\
 Germany}
 \email{dehsani.math@gmail.com}


\begin{abstract}
  We study the Dirichlet to Neumann operator 
 of the \dbar-Neumann problem, and the relation
  between the \dbar-Neumann 
  boundary conditions and the 
  Dirichlet to Neumann operator.
\end{abstract}

\maketitle

\section{Introduction}
\label{intro}

The \dbar-Neumann problem is an example of a
 boundary value problem with involving an elliptic 
operator but whose boundary conditions lead to 
 non-elliptic equations.  In order to conclude Sobolev
estimates for the solution to the \dbar-Neumann problem
 control (of $L^2$-norms) over derivatives
  in all directions must be
obtained, but the boundary conditions of the problem
 disadvantage one direction.  The boundary 
conditions contain the boundary value operator,
 the Dirichlet to Neumann operator (DNO), giving the
boundary values of the outward derivative of the
 solution to a homogeneous Dirichlet problem.  
In some cases (for example the case of strictly
 pseudoconvex domains) the DNO allows for 
some control of the disadvantaged direction, in other
 cases of weak pseudoconvexity, the situation is
  more delicate.    
 The purpose of this article therefore, is to 
study the DNO of related to the \dbar-Neumann problem
 with particular emphasis on the resulting boundary equations.

The DNO will be written as a pseudodifferential 
 operator acting on a boundary distribution, and
our first results are a reworking of results of 
 Chang, Nagel, and Stein in \cite{CNS92}.
It is well known that to highest order the 
DNO is given by the square root
of the highest order tangential terms in the elliptic interior
operator.  The highest two orders of the DNO are calculated, as in
\cite{CNS92}, and reduce to those results in a special case.  The
approach of \cite{CNS92} could be used here as well to calculate
the DNO, but we take another approach 
outlined in \cite{Eh18_halfPlanes} 
based on pseudodifferential
operators on domains with boundary, an approach which  
was useful in 
calculating the symbol of the normal derivative to the Green's
operator, as well permitting similar 
calculations and estimates in the situation of
piecewise smooth domains \cite{Eh18_pwSmth}.
 Relations among operators comprising the DNO, as
well as other derived boundary value operators in the
 boundary conditions are essential in the 
construction of a solution to the $\mdbar$-problem
 if a solution operator to $\mdbar_b$ is assumed
in \cite{Eh18_dbarDbarb}.

 We further demonstrate in this paper the
  persistence of the non-elliptic character
of the \dbar-Neumann conditions.  In Section
 \ref{varSquare}, we examine what 
 happens when a perturbation is made of the 
  elliptic operator of the problem.  This change naturally
also leads to a different DNO, however 
 as we shall see the associated boundary condition
  is essentially the same (and non-elliptic!).
The boundary
operator can be approximated by
Kohn's Laplacian, $\square_b$.  This suggests 
that the \dbar-Neumann problem can
be solved by inverting the 
$\square_b$ operator and that 
the \dbar-problem can be solved by
using a solution operator for
$\mdbar_b$.  This approach to 
$\mdbar$ is taken up in 
\cite{Eh18_dbarDbarb}.

Most the work presented here was
 undertaken while the author was at the
University of Wuppertal and the hospitality of the University and
its Complex Analysis Working Group is sincerely appreciated.  The
author particularly thanks Jean Ruppenthal for his warm and
generous invitation to work with his group.  A visit to the
Oberwolfach Research Institute in 2013 as part of a Research in
Pairs group was also helpful in the formation of this article, for
which the author extends gratitude to the Institute as well as to
S\"{o}nmez \c{S}ahuto\u{g}lu for helpful discussions.

\section{Notation and background}
\label{notation}

We fix some
  notation used throughout the article.
 Our notation for derivatives is $\partial_t:=\frac{\partial}{\partial t}$.
  We also use the index notation for derivatives: with
  $\alpha=(\alpha_1,\ldots,\alpha_n)$ a multi-index
\begin{equation*}
 \partial^{\alpha}_x = \partial_{x_1}^{\alpha_1}\cdots
  \partial_{x_n}^{\alpha_n}.
\end{equation*}
 Multiplication of derivatives with
$-i$ come in handy when dealing with symbol
 expansions of pseudodifferential operators and
we will use the notation
$D_x^{\alpha}$ to denote 
 $-i \partial_x^{\alpha}$.

 We let $\Omega\subset\mathbb{R}^n$ be a smoothly bounded 
  domain and
  define pseudodifferential operators on $\Omega$ as in \cite{Tr}:
\begin{defn}
 We denote by
 $\mathcal{S}^{\alpha}(\Omega)$
the space of symbols
 $a(x,\xi)\in C^{\infty}(\Omega\times\mathbb{R}^n)$ which
 have the property that for any given compact set, $K$, and for
 any $n$- tuples $k_1$ and $k_2$, there is a
 constant $c_{k_1,k_2}(K)>0$ such that
\begin{equation*}
 \left| \partial_{\xi}^{k_1} \partial_{x}^{k_2} a(x,\xi) \right|
  \le c_{k_1,k_2}(K) \left( 1+|\xi|
   \right)^{\alpha-|k_1|}
    \qquad \forall x\in K,\ \xi\in\mathbb{R}^n.
\end{equation*}
\end{defn}
Associated to the symbols in class $\mathcal{S}^{\alpha}(\Omega)$
are the pseudodifferential operators, denoted by
$\Psi^{\alpha}(\Omega)$.
 If $u\in \mathscr{E}'(\Omega)$, we can
define $u\in \mathscr{E}'(\mathbb{R}^n)$ by 
 using an extension by 0, and then define
the Fourier Transform of the extended $u$.
We denote the transform of the extended 
 distribution simply by $\widehat{u}(\xi)$.
The definition of pseudodifferential operators
 on a domain $\Omega$ is given by
\begin{defn}
 \label{defnop}
 We say an operator $A: \mathscr{E}'(\Omega)\rightarrow
  \mathscr{D}'(\Omega)$ is in class
  $\Psi^{\alpha}(\Omega)$
  if $A$ can be written as an integral operator with symbol
  $a(x,\xi)\in \mathcal{S}^{\alpha}
   (\Omega)$:
\begin{equation}
 \label{defnom}
 Au (x) = \frac{1}{(2\pi)^n} \int_{\mathbb{R}^n} a(x,\xi)
 \widehat{u}(\xi) e^{ix\cdot \xi} d\xi.
\end{equation}
\end{defn}

In our applications in this article we
 will be dealing with operators defined on
all of $\mathbb{R}^{2n}$ applied to functions
 defined on $\Omega$ (but which can be 
  extended by 0 to the whole space).  The
operators on $\Omega$ will thus be the composition of
 the restriction to $\Omega$ operator with the
pseudodifferential operators defined on 
 $\mathbb{R}^{2n}$.

 If we let
$\chi_j$ be such that $\{\chi_j\equiv 1\}_j$ is a covering of
$\Omega$, and let $\varphi_j$ be a partition of unity subordinate
to this covering, then locally, we describe a boundary operator
$A: \mathscr{E}'(\Omega)\rightarrow
\mathscr{D}'(\Omega)$ in terms of its symbol,
$a(x,\xi)$ according to
\begin{equation*}
Au = \frac{1}{(2\pi)^n} \int a(x,\xi) \widehat{\chi_j
	u}(\xi)
d\xi
\end{equation*}
on $\mbox{supp }\varphi_j$.  Then we can describe the operator $A$
globally on all of $\Omega$ by
\begin{equation}
\label{alld}
Au = \frac{1}{(2\pi)^n} \sum_j \varphi_j \int a(x,\xi) \widehat{\chi_j
	u}(\xi)
d\xi.
\end{equation}
The difference arising between the definitions in \eqref{defnom}
and \eqref{alld} is a smoothing term \cite{Tr}, which we write as
$\Psi^{-\infty}u$, to use the notation of Definition
\ref{defnop}.

While $\Psi^{\alpha}(\Omega)$ will denote a class of
 operators, the use of $\Psi^{\alpha}$ will 
be used to refer to any operator in class
 $\Psi^{\alpha}(\Omega)$.
Furthermore, operators defined on the boundary of a domain
 will be denoted with a subscript $b$.  For instance,
if $A\in \Psi^{\alpha}(\partial\Omega)$ we write
 $A=\Psi_b^{\alpha}$.

In our use of Fourier transforms and equivalent symbols we use
cutoffs in order to make use of local
coordinates, one of
which being a defining function,
denoted by $\rho$, for the domain. 
We use $\widetilde{\ \ }$ to indicate transforms in tangential
directions.  Let $p\in\partial\Omega$ and let
$(x_1,\ldots,x_{n-1},\rho)$ be local coordinates around $p$,
$(\rho<0)$.  Let $\chi_p(x,\rho)$ denote a cutoff which is $\equiv
1$ near $p$ and vanishes outside a small neighborhood of $p$ on
which the local coordinates $(x,\rho)$ are valid.  Then with $u\in
L^2(\Omega)$ we write
\begin{align*}
 & \widehat{\chi_p u}(\xi,\eta)=
  \int \chi_p u(x,\rho) e^{-ix\xi} e^{-i\rho\eta} dx d\rho\\
 &\widetilde{\chi_p u}(\xi,\rho)=
  \int \chi_p u(x,\rho) e^{-ix\xi}  dx.
\end{align*}
We also use the $\widetilde{\ \ }$ notation when describing
transforms of functions supported on the boundary.  With notation
and coordinates as above, we let $u_b(x)\in L^2(\partial\Omega)$
and write
\begin{equation*}
\widetilde{\chi_p(x,0) u_b}(\xi)=
  \int \chi_p(x,0) u_b(x) e^{-ix\xi}  dx.
\end{equation*}

We want to apply
pseudodifferential operator techniques to 
vector fields on a smoothly bounded domain
  $\Omega\subset \mathbb{C}^n$.  Let 
$\rho$ be a smooth defining function for
 $\Omega$ ($\Omega =
  \{z\in C^n: \rho(z) <0\}$), normalized so that
$|\nabla \rho| =1$ on $\partial\Omega$.
  We choose an 
orthonormal basis of $(1,0)$ forms,
$\omega_1,\ldots, \omega_n$ in which 
$\omega_n=\sqrt{2}\partial\rho$,
 and we denote $L_1,\ldots, L_n$ the 
vector fields respectively dual 
 to the $\omega_j$.

We let $T = \frac{1}{2i} (L_n - \Lb_n)$, and
 $T^0 = T|_{\partial\Omega}$.
If we choose a boundary point $p$
 we can choose local coordinates, as above, in a 
neighborhood of $p$ such that $L_n$ has the
 form
\begin{align*}
\nonumber
L_n =& \frac{1}{\sqrt{2}}
\frac{\partial}{\partial \rho} + i T\\
\nonumber
=&\frac{1}{\sqrt{2}}
\frac{\partial}{\partial \rho} +
i T^0 + O(\rho)\\
=& \frac{1}{\sqrt{2}}
\frac{\partial}{\partial \rho} +
i\frac{\partial}{\partial x_{2n-1}} + O(\rho).
\end{align*}

Similarly, in a neighborhood of $p$, we can
 represent the $L_j$ vector fields as
\begin{equation}
\label{Ljp}
L_j= \frac{1}{2}\left(
\frac{\partial}{\partial x_{2j-1}}
-i \frac{\partial}{\partial x_{2j}}
\right) + \sum_{k=1}^{2n-1}\ell^j_k(x-p)
\frac{\partial}{\partial x_{k}}
+O(\rho),
\end{equation}
where $\ell^j_k(x)=O(x)$.  
 
In Fourier space we use $\xi$ to denote the
 dual variables to the $x$ coordinates,
$\xi_i$ corresponding to $x_i$ for 
 $i=1,\ldots, 2n-1$, and $\eta$ dual to 
$\rho$.  To help 
 distinguish the complex tangential
behavior, we set
 $\xi_L^2$ to be given by
\begin{equation*}
\xi_L^2=\sum_1^{2n-2} \xi_j^2.
\end{equation*}
We use the standard decomposition of the Fourier transform space
to separate three microlocal neighborhoods (see
for instance
\cite{Ch91, KoNi06, K85, Ni06}).  We let
$\psi^{+}$, $\psi^{0}$, and $\psi^-$ be a smooth partition of
unity on the unit ball, $|\xi|=1$.  We choose the functions so
that $\psi^+$ has support in $\xi_{2n-1}>\frac{1}{2}|\xi_L|$
and $\psi^+\equiv 1$ in the
region $\xi_{2n-1}>\frac{3}{4}|\xi_L|$.  The function
$\psi^-$ is defined symmetrically: $\psi^-$
has support in $\xi_{2n-1}<-\frac{1}{2}|\xi_L|$
and $\psi^-\equiv 1$ in the
region $\xi_{2n-1}<-\frac{3}{4}|\xi_L|$.  The function
$\psi^0$ has support in 
$|\xi_{2n-1}|< \frac{3}{4}|\xi_L|$ and satisfies
$\psi^0=1-\psi^+ -\psi^-$.  We extend the functions radially, so that, in particular, they satisfy 
\begin{equation*}
|\partial_{\xi}^k \psi^{\ast} |
\lesssim |\xi|^{-k}
\end{equation*} 
outside of some compact neighborhood of $\xi=0$.
This last property ensures that
$\psi^{+}$, $\psi^{0}$, and $\psi^-$ are in the class of symbols, 
$\lrs^0(\mathbb{R}^{2n-1})$. 
Cutoffs are also introduced so that (using the same notation
for the functions) $\psi^0\equiv 1$ on a 
neighborhood of $\xi=0$ contained in $|\xi|<1$.
The radial extensions from the unit circle together with
the support of $\psi^0$ near 0 
are then to
form a partition of unity of the transform space,
i.e., $\psi^+ +\psi^0 +\psi^-=1$ for all $\xi\in \mathbb{R}^{2n-1}$.
The operators
corresponding to the symbols,
$\psi^+$, $\psi^0$, and $\psi^-$,
will be denoted by
$\Psi^{\nu^+}$, $\Psi^{\nu^0}$, and $\Psi^{\nu^-}$, respectively.

As mentioned above, we take an approach to
 calculating boundary value operators based on
a pseudodifferential calculus for domains 
 with boundary worked out in \cite{Eh18_halfPlanes}.  
In particular, we will make use of the results
 detailing the behavior of functions which
result from the application of pseudodifferential
 operators (in $\mathbb{R}^{2n}$) to distributions
supported on the boundary of the domain, as well
 as certain operators applied to distributions
with support in the whole domain 
 (which can be thought of as a distribution
  on all of $\mathbb{R}^{2n}$ with an 
extension by 0).  Let us recall here a few results
 from \cite{Eh18_halfPlanes}.
The results in \cite{Eh18_halfPlanes} were stated for
 half-planes and these will be applied to 
domains $\Omega\subset \mathbb{C}^n \simeq \mathbb{R}^{2n}$,
 using local coordinates
$(x,\rho)$ with $\rho<0$ defining the domain.

We first define certain operators
 which appear in taking inverses to elliptic operators:
\begin{defn}
	\label{defnDecomp}
 Let $A\in\Psi^{-k}(\mathbb{R}^{2n})$ for $k\ge 1$
 have the property that 
  for any $N\in \mathbb{N}$, it
can be written in the form
\begin{equation*}
 A= B + \Psi^{-N},
\end{equation*}
where $B\in\Psi^{-k}(\mathbb{R}^{2n})$ has symbol,
$\sigma(B)(x,\rho,\xi,\eta)$, which
is meromorphic (in $\eta$) with poles at 
\begin{equation*}
\eta=q_1(x,\rho,\xi), \ldots, q_k(x,\rho,\xi)
\end{equation*}	
with $q_i(x,\rho,\xi)$ themselves, as well
 as the imaginary parts,
$\mbox{Im }q_i$,
symbols of 
pseudodifferential operators of
order 1 (restricted to $\eta=0$)
such that for each 
$\rho$, $\mbox{Res}_{\eta=q_i}
\sigma(B) \in \mathcal{S}^{k+1}
(\mathbb{R}^{2n-1})$ with symbol estimates
uniform in the $\rho$ parameter.

We call such an operator, $A$, a {\it decomposable} 
 operator.
\end{defn}

 The first theorem is taken from 
Theorems 2.2 and 2.4 of \cite{Eh18_halfPlanes}.
\begin{thrm}
\label{estinvell}
Let $g\in \mathscr{D}(\Omega)$ of
the form $g(x,\rho) =g_b(x) \delta(\rho)$ for
	$g_b\in W^s(\partial\Omega)$.  Let
$A\in \Psi^{k}(\Omega)$, $k\le -1$ be
a decomposable operator.
Then for all $s$,
\begin{equation*}
\|A g\|_{W^s(\Omega)}
  \lesssim \|
g_b\|_{W^{s+k+1/2}(\partial\Omega)}.
\end{equation*}
\end{thrm}

Theorem \ref{estinvell} for instance is 
 applicable for any term arising in the symbol
expansion of 
 the inverse to an elliptic differential
operator.  With $A$ an elliptic differential
 operator of order $k$, we can write
for any $N\in \mathbb{N}$,
\begin{equation}
 \label{formInv}
  A^{-1} = B_{-k} + \Psi^{-N}
\end{equation}
where $B_{-k}$ is a pseudodifferential operator
 of order $-k$ and satisfies the conditions of
  the $B$ operator in 
Definition \ref{defnDecomp}.  

Thus for instance if 
 $\triangle$ is a second order elliptic differential
 operator on $\mathbb{R}^{2n}$, and 
 for some given $s\ge0$,
$g_b \in W^s(\partial\Omega)$, 
 with $g = g_b \times \delta(\rho)$ as above.
   Then 
$\triangle^{-1}$ satisfies the hypothesis of 
 Theorem \ref{estinvell} and we have
\begin{equation*}
\|\triangle^{-1} g\|_{W^{s+3/2}(\Omega)}
\lesssim \|
g_b\|_{W^{s}(\partial\Omega)}.
\end{equation*}

We also have the following useful Lemmas.
\begin{lemma}
\label{liglem}
Let $g\in \mathscr{D}(\Omega)$ be of the
 form $g(x,\rho) =g_b(x) \delta(\rho)$ for
 $g_b\in \mathscr{D}(\partial \Omega)$ .  Let
$A\in \Psi^k(\Omega)$, be a
	pseudodifferential operator of order $k$.
	Let $\rho$ denote the operator of
	multiplication with $\rho$. Then
	$\rho\circ A$ induces a pseudodifferential 
	operator of order $k-1$
	on $g$:
\begin{equation*}
	\rho A g \equiv \Psi^{k-1} g.
\end{equation*}
\end{lemma}

Let $R$ denote the restriction operator, $R:
\mathscr{D}(\Omega) \rightarrow \mathscr{D}(\partial\Omega)$,
given in local coordinates
 $(x,\rho)$ by
$R\phi=\left.\phi\right|_{\rho=0}$.
\begin{lemma}
\label{restrict} 
Let $g\in \mathscr{D}(\Omega)$ be of the
form $g(x,\rho) =g_b(x) \delta(\rho)$ for
$g_b\in \mathscr{D}(\partial\Omega)$ .  Let
$A\in \Psi^k(\Omega)$, 
 be an operator of order $k$,
	for $k\le -2$.  
Then $R \circ A$ induces a pseudodifferential
	operator in 
$\Psi^{k+1}_b(\partial\Omega)$ 
	acting on $g_b$ via
\begin{equation*}
	R\circ A g \equiv \Psi^{k+1}_b g_b.
\end{equation*}
\end{lemma}

\section{\dbar-Neumann problem}
\label{dbarN}

We look more closely at the \dbar-Neumann problem,
$\square u = f$, where
\begin{equation*}
\square =\mdbar\mdbar^{\ast} + \mdbar^{\ast}\mdbar
.
\end{equation*}
For $f$ a $(0,q)$-form, the equation
$\square u = f$ comprises a system of equations,
and we write our equations in matrix form.  
We use the convention of writing indices with
increasing entries: a particular index of
length $q$,
$J=(j_1,\ldots, j_q)$, is ordered according to
$j_l<j_m$ for $l<m$.
For the matrix
we consider the ordering of two indices,
$J_1 = (j_{11},j_{12},\ldots, j_{1q})$ and
$J_2 = (j_{21},j_{22},\ldots, j_{2q})$,
according to
$J_1< J_2$ if $j_{1k}< j_{2k}$ for the first 
 $k$ such that $j_{1k} \neq j_{2k}$, 
 and $J_1= J_2$ if $j_{1k}= j_{2k}$ for
all $k=1,\ldots q$.  The rows (and columns) of the matrix
are 
in order of increasing indices.  Thus, for instance,
if we denote 
$J_1=(1,2,\ldots, q)$,
the $(1,1)$-entry of the matrix corresponds to the action
on $u_{J_1}$ which results in a form whose component is
$\omegab_{J_1}$.  Similarly, with 
$J_2=(1,2,\ldots,q-1,n)$, the $(n-q+1,1)$-entry of the matrix
corresponds to the action
on $u_{J_1}$ which results in a form whose component is
$\omegab_{J_2}$, etc.

We want to calculate in general a $J^{th}$ row of the matrix
of operators describing $\square$.
  We thus need to know which forms would result
 in a $\omegab_J$ term
when some input form is given into
 $\square$.  
Let $J=(j_1, \ldots, j_q)$ with
$j_m=k$, where $1\le k\le n$, for some $m$.
We use the notation $J_{\hat{k}}$ to denote
the index of length $q-1$
$(j_1,\ldots,j_{m-1},j_{m+1},\ldots,j_q)$.  We further
use the set notation $J_{\hat{k}}\cup\{l\}$ to denote 
the index of length $q$ 
 (we assume the case $l\neq j_i$ for
  $1\le i\le q$, $i\neq m$) consisting of
the set $\{
j_1,\ldots,j_{m-1},j_{m+1},\ldots,j_q,l
\}$ in the appropriate order (recall 
a particular index
$K=(k_1,\ldots, k_q)$ is ordered if
$k_r<k_s$ for $r<s$).

Since $\mdbar$ increases the type of the form by 1 and
$\mdbar^{\ast}$ decreases it by one, we see
it is when $\square$ operates on a form of
the type $u_{J_{\hat{k}}\cup\{l\}}
\omegab_{J_{\hat{k}}\cup\{l\}}$ that a
$\omegab_J$ term would result.  And so we calculate    $\square u_{J_{\hat{k}}\cup\{l\}}
\omegab_{J_{\hat{k}}\cup\{l\}}$.
With abuse of notation, for 
 a prescribed $J$, we write
$u_{kl}$ in place of
$u_{J_{\hat{k}}\cup\{l\}}$, 
assuming $k\neq l$, with the obvious
 simplification in dimension 2.

We use our notation in the case $k\neq l$, 
 although the case $k=l$ is also included in 
the same calculations.
We start with the $\omegab_{J_{\hat{k}}}$ components 
resulting from 
$\mdbar^{\ast} \left( u_{kl}\omegab_{J_{\hat{k}}\cup\{l\}}
\right)$. 
We use the notation
\begin{equation}
\label{cCoeff}
c_{J\cup \{m\}}^J= \mdbar(\omegab_J) \rfloor
\omegab_{J\cup \{m\}},
\end{equation}
and for $j=1,\ldots,n$,
we define $d_j$ according to
an integration by parts
\begin{equation*}
(\phi,\Lb_j\varphi)= 
\Big( (-L_j + d_j)\phi,\varphi\Big)
\end{equation*}
for $\phi, \varphi \in C^{\infty}_0(\Omega)$ 
with support in a coordinate patch such that
so that $L_j$ can be written in terms of local coordinates 
as in \eqref{Ljp}.

From
\begin{align*}
\nonumber
\left(
\mdbar^{\ast} u_{kl}\omegab_{J_{\hat{k}}\cup\{l\}},
\varphi \omegab_{J_{\hat{k}}}
\right) =&
\left(
u_{kl}\omegab_{J_{\hat{k}}\cup\{l\}},
\mdbar( \varphi \omegab_{J_{\hat{k}}})
\right)\\
\nonumber
=& \left(
u_{kl}\omegab_{J_{\hat{k}}\cup\{l\}},
\Lb_l \varphi \omegab_l\wedge \omegab_{J_{\hat{k}}} + c^{J_{\hat{k}}}_{J_{\hat{k}}\cup\{l\}}\varphi
\omegab_{J_{\hat{k}}\cup\{l\}} \right)\\
=& \left(
u_{kl},
\varepsilon^{lJ_{\hat{k}} }_{J_{\hat{k}}\cup \{l\}}
\Lb_l \varphi  + c^{J_{\hat{k}}}_{J_{\hat{k}}\cup\{l\}}\varphi
\right),
\end{align*}
where we write,
for some index $K=(k_1,\ldots k_{q-1})$ and
$mK = (m,k_1,\ldots k_{q-1})$, 
\begin{equation*}
\varepsilon^{mK}_{K\cup \{m\}}
= \omegab_{mK}  \rfloor
\omegab_{K\cup \{m\}},
\end{equation*}
we have
\begin{equation}
\label{dstarJk}
\mdbar^{\ast} \left( u_{kl}\omegab_{J_{\hat{k}}\cup\{l\}}
\right)
= \left( \varepsilon^{lJ_{\hat{k}} }_{J_{\hat{k}}\cup \{l\}} \left(
-L_l + d_l\right)u_{kl}
+ \overline{c}^{J_{\hat{k}}}_{J_{\hat{k}}\cup\{l\}}u_{kl}\right) \omegab_{J_{\hat{k}}}
+ \cdots,
\end{equation}
with $\varphi$ some test function,
where the $\cdots$ refers to terms whose contraction with
$\omegab_{J_{\hat{k}}}$ results in 0
(and which contain
a $\omegab_l$ component).
And thus
\begin{align}
\nonumber
\mdbar \mdbar^{\ast} u_{kl}\omegab_{J_{\hat{k}}\cup\{l\}}
=& \left( \varepsilon^{lJ_{\hat{k}} }_{J_{\hat{k}}\cup \{l\}}
\Lb_k \left(
-L_l + d_l\right)u_{kl}
+ \overline{c}^{J_{\hat{k}}}_{J_{\hat{k}}\cup\{l\}}\Lb_k u_{kl}\right) \omegab_k\wedge \omegab_{J_{\hat{k}}}\\
\nonumber
& -  \varepsilon^{lJ_{\hat{k}} }_{J_{\hat{k}}\cup \{l\}}
c^{J_{\hat{k}}}_{J}
L_l u_{kl}\omegab_J
+ \cdots\\
\nonumber
=& \Big(- \varepsilon^{lJ_{\hat{k}} }_{J_{\hat{k}}\cup \{l\}} \varepsilon^{kJ_{\hat{k}} }_{J} \Lb_k L_l
u_{kl}
+ \big(\varepsilon^{lJ_{\hat{k}} }_{J_{\hat{k}}\cup \{l\}} \varepsilon^{kJ_{\hat{k}} }_{J}d_l
+ \varepsilon^{kJ_{\hat{k}} }_{J}
\overline{c}^{J_{\hat{k}}}_{J_{\hat{k}}\cup\{l\}}\big) \Lb_k u_{kl}\\
\label{ddstar}
&- \varepsilon^{lJ_{\hat{k}} }_{J_{\hat{k}}\cup \{l\}}
c^{J_{\hat{k}}}_{J}
L_l u_{kl} \Big)\omegab_J + \cdots,
\end{align}
where here the $\cdots$ refers to terms which
upon contraction with $\omegab_J$ result in 0
as well as zero order terms.

We note that the calculations also show
\begin{equation}
\label{ddiag}
\mdbar \mdbar^{\ast} u_{J}\omegab_{J}
= \sum_{l\in J} \Big(- \Lb_l L_l
u_{J}
+ \big(d_l
+ \varepsilon^{lJ_{\hat{l}} }_{J}
\overline{c}^{J_{\hat{l}}}_{J}\big) \Lb_l u_{J}- \varepsilon^{lJ_{\hat{l}}} _{J}
c^{J_{\hat{l}}}_{J}
L_l u_{J} \Big)\omegab_J + \cdots.
\end{equation}

Similarly, to calculate
$ \mdbar^{\ast}\mdbar u_{kl}\omegab_{J_{\hat{k}}\cup\{l\}}$ we start with
\begin{equation*}
\mdbar u_{kl}\omegab_{J_{\hat{k}}\cup\{l\}}
= \left( \varepsilon^{kJ_{\hat{k}}\cup \{l\}}_{J\cup \{l\}} \Lb_k u_{kl}
+ c^{J_{\hat{k}}\cup\{l\}}_{J\cup\{l\}}
u_{kl}  \right)\omegab_{J\cup\{l\}}
\end{equation*}
modulo terms whose contraction with
$\omegab_{J\cup\{l\}}$ result in 0.
As in \eqref{dstarJk}, we have
\begin{equation*}
\mdbar^{\ast} v \omegab_{J\cup\{l\}}
= \left( \varepsilon^{lJ}_{J\cup\{l\}}(-L_l+d_l )v
+ \overline{c}^{J}_{J\cup\{l\}}v\right) \omegab_J
\end{equation*}
modulo terms whose contraction with
$\omegab_{J}$ result in 0,
which when applied to $\mdbar u_{kl}\omegab_{J_{\hat{k}}\cup\{l\}}$ above, yields
\begin{align}
\nonumber
\mdbar^{\ast}\mdbar u_{kl}\omegab_{J_{\hat{k}}\cup\{l\}}=&   \Big( -\varepsilon^{lJ}_{J\cup\{l\}}
\varepsilon^{kJ_{\hat{k}}\cup \{l\}}_{J\cup \{l\}} L_l \Lb_k u_{kl}\\
\nonumber
&+ 
\varepsilon^{lJ}_{J\cup\{l\}}
\varepsilon^{kJ_{\hat{k}}\cup \{l\}}_{J\cup \{l\}} d_l  \Lb_k u_{kl}
+ \varepsilon^{kJ_{\hat{k}}\cup \{l\}}_{J\cup \{l\}}  \overline{c}^{J}_{J\cup\{l\}}
\Lb_k u_{kl}
\\
\label{dstard}
&- \varepsilon^{lJ}_{J\cup\{l\}}
c^{J_{\hat{k}}\cup\{l\}}_{J\cup\{l\}}
L_l u_{kl}\Big) \omegab_J + \cdots,
\end{align}
where the $\ldots$ refers to terms 
whose contraction with
$\omegab_{J}$ result in 0 as well as
terms of order 0.
And similarly,
\begin{align}
\nonumber
\mdbar^{\ast} \mdbar u_{J}\omegab_{J}
=& \sum_{l\notin J} \Big(-  L_l\Lb_l
u_{J}
+ \big(d_l
+ \varepsilon^{lJ }_{J\cup \{ l \}}
\overline{c}^{J}_{J\cup \{ l \}}\big) \Lb_l u_{J}\\
\label{dstardiag}
&- \varepsilon^{lJ} _{J\cup \{ l \}}
c^{J}_{J\cup \{ l \}}
L_l u_{J} \Big)\omegab_J + \cdots.
\end{align}

Adding \eqref{ddiag} and
\eqref{dstardiag} yields
$\square u_J \omegab_J$:
\begin{align*}
\square u_J \omegab_J&=
-\sum_{l\in J} \Lb_l L_l u_{J}\omegab_J
- \sum_{l\notin J}  L_l\Lb_l u_{J}\omegab_J
\\
+&
\begin{cases}
(-1)^{|J|} \left(
- \overline{c}_{J}^{J_{\hat{n}}} \Lb_n u_{J} \omegab_J
+ c_{J}^{J_{\hat{n}}} L_n \right) u_J \omegab_J
+  d_n
\Lb_n u_J \omegab_J + \cdots
&n\in J
\\
(-1)^{|J|} \left(
\overline{c}_{J\cup\{n\}}^J \Lb_n
-
c_{J\cup\{n\}}^J L_n \right) u_J \omegab_J
+  d_n
\Lb_n u_J \omegab_J + \cdots
&n\notin J.
\end{cases}
\end{align*}
We only collect the 
(complex) normal derivatives, as 
they are enough to determine the behavior of the
relevant boundary value operators in the 
direction of the field
$T$ (see the discussion at the end of 
Section \ref{SecZero}).  
The terms included in 
 the $\cdots$ 
thus refer to terms which are orthogonal to
$\omegab_J$, of zero order, or involve only vector fields
orthogonal to $L_n$ or $\Lb_n$.  

The two cases can be combined into
the single expression
\begin{align*}
\square u_J \omegab_J&=
-\sum_{l\in J} \Lb_l L_l u_{J}\omegab_J
- \sum_{l\notin J}  L_l\Lb_l u_{J}\omegab_J
\\
+& (-1)^{|J\cup \{ n\}|}  \left(
c^{J\setminus \{n\}}_{J\setminus \{n\} \cup \{ n\}} L_n
-\overline{c}^{J\setminus \{n\}}_{J\setminus \{n\} \cup \{ n\}} \Lb_n
\right) u_J \omegab_J
+  d_n
\Lb_n u_J \omegab_J
+\cdots .
\end{align*}

We now add \eqref{ddstar} and \eqref{dstard}
to obtain $\square u_{kl}\omegab_{J_{\hat{k}}\cup\{l\}}$.  To simplify the result we note
\begin{equation}
\label{simpleVarEpsilon}
\varepsilon^{lJ}_{J\cup\{l\}}
\varepsilon^{kJ_{\hat{k}}\cup \{l\}}_{J\cup \{l\}} =
- \varepsilon^{lJ_{\hat{k}} }_{J_{\hat{k}}\cup \{l\}} \varepsilon^{kJ_{\hat{k}} }_{J}
\end{equation}
for $k\ne l$,
and as we are only interested in
the complex normal derivatives,
in comparing
$\varepsilon^{lJ}_{J\cup\{l\}}
c^{J_{\hat{k}}\cup\{l\}}_{J\cup\{l\}} $ and $ \varepsilon^{lJ_{\hat{k}} }_{J_{\hat{k}}\cup \{l\}}
c^{J_{\hat{k}}}_{J} $
we look at the case $l=n$, for which we have
\begin{equation*}
\varepsilon^{nJ}_{J\cup\{n\}}
c^{J_{\hat{k}}\cup\{n\}}_{J\cup\{n\}}
= -  \varepsilon^{nJ_{\hat{k}} }_{J_{\hat{k}}\cup \{n\}}
c^{J_{\hat{k}}}_{J}.
\end{equation*}
Similarly, in comparing
$\varepsilon^{kJ_{\hat{k}}\cup \{l\}}_{J\cup \{l\}}  \overline{c}^{J}_{J\cup\{l\}}$ and
$ \varepsilon^{kJ_{\hat{k}} }_{J}
\overline{c}^{J_{\hat{k}}}_{J_{\hat{k}}\cup\{l\}}$ we are only interested in the case
$k=n$ for which
\begin{equation*}
\varepsilon^{nJ_{\hat{n}}\cup \{l\}}_{J\cup \{l\}}  \overline{c}^{J}_{J\cup\{l\}}
= -  \varepsilon^{nJ_{\hat{n}} }_{J}
\overline{c}^{J_{\hat{n}}}_{J_{\hat{n}}\cup\{l\}}.
\end{equation*}
We thus can write, by adding \eqref{ddstar}
and \eqref{dstard},
\begin{equation*}
\square u_{kl}\omegab_{J_{\hat{k}}\cup\{l\}}=
- \varepsilon^{lJ_{\hat{k}} }_{J_{\hat{k}}\cup \{l\}} \varepsilon^{kJ_{\hat{k}} }_{J}
[ \Lb_k, L_l]
u_{kl}\omegab_J 
\end{equation*}
for $k\ne l$, modulo terms 
 with the vector fields
$L_j$ or $\Lb_j$ for $j=1,\ldots n-1$,
zero order terms, or forms orthogonal to 
$\omegab_J$.

We collect our results in the following proposition:
\begin{prop}
	\label{squareOp}
	Modulo the vector fields
	$L_j$ or $\Lb_j$ 
acting on components of 
 $u$ for $j=1,\ldots n-1$,
	zero order terms, or forms orthogonal to 
	$\omegab_J$, we have
\begin{align*}
	& i)\ \square \left(
	u_{J}  \omegab_{J}
	\right) =
	-\sum_{l\in J} \Lb_l L_l u_{J}\omegab_J
	- \sum_{l\notin J}  L_l\Lb_l u_{J}\omegab_J\\
	&\qquad \qquad
	+ (-1)^{|J\cup \{ n\}|}  \left(
	c^{J\setminus \{n\}}_{J\setminus \{n\} \cup \{ n\}} L_n
	-\overline{c}^{J\setminus \{n\}}_{J\setminus \{n\} \cup \{ n\}} \Lb_n
	\right) u_J \omegab_J
	+  d_n
	\Lb_n u_J \omegab_J
	\\
	& ii)\ \square u_{J_{\hat{k}}l}
	\omegab_{J_{\hat{k}}\cup\{l\}}=
	- \varepsilon^{lJ_{\hat{k}} }_{J_{\hat{k}}\cup \{l\}} \varepsilon^{kJ_{\hat{k}} }_{J}
	[ \Lb_k, L_l]
	u_{kl}\omegab_J.
	\end{align*}
\end{prop}

\section{The Dirichlet to Neumann operator}
\label{dnosec} 

The Dirichlet to Neumann operator (DNO) is the
boundary value operator giving the outward
 normal derivative of the
solution to a Dirichlet problem.  We look
 at the DNO
corresponding to the operator $2\square$.
 We study the solution, $v$, 
 which solves
\begin{align*}
2 \square v =& 0 \qquad \mbox{in }
   \Omega\\
v=& g_b \qquad \mbox{on }
   \partial \Omega,
\end{align*}
and we obtain an expression for
 $\frac{\partial v}{\partial \rho}$ 
  (modulo smooth terms) near a given point
 $p\in\partial\Omega$ in terms of $g_b$.

 If $\chi_p$ is a smooth cutoff
 function with support in a small neighborhood of $p$
  and $\chi_p'$ a smooth cutoff such that $\chi_p' \equiv 1$ on
   $\mbox{supp }\chi_p$, we have
\begin{align}
 \label{2bgl}
2\square (\chi_p v) =&   \Psi^1 (\chi_p' v) \qquad \mbox{on }
   \Omega\\
\chi_p v=& \chi_p g_b \qquad \mbox{on }
   \partial \Omega.
\nonumber
\end{align}
 The term $\Psi^1 (\chi_p' v)$ above arises due to derivatives
 falling on the cutoff function $\chi_p$.  The use of the cutoff
 function allows us to consider the equation locally, and is
 equivalent to using pseudodifferential operators with symbols
 defined in local coordinate patches, with one of the coordinates
 given by $\rho$.
We thus consider $v$ to be 
 supported in a neighborhood of a given
boundary point, $p\in\partial\Omega$.

 To study the operator $\square$ and the associated boundary
operators, we consider the equation $\square v =0$ for 
a $(0,q)$-form $v$.  
In a small neighborhood of a boundary point, which we take to 
be $0\in\partial\Omega$, we write the vector
fields, $L_j$ in local coordinates as 
in \eqref{Ljp}:
\begin{equation}
 \label{Lj}
L_j= \frac{1}{2}\left(
\frac{\partial}{\partial x_{2j-1}}
-i \frac{\partial}{\partial x_{2j}}
\right) + \sum_{k=1}^{2n-1}\ell^j_k(x)
\frac{\partial}{\partial x_{k}}
+O(\rho),
\end{equation}
where $\ell^j_k(x)=O(x)$.  
Also we recall from Section \ref{notation},
the representation of $L_n$:
\begin{equation}
L_n =
 \label{LnRepeat}
\frac{1}{\sqrt{2}}
\frac{\partial}{\partial \rho} +
i\frac{\partial}{\partial x_{2n-1}} + O(\rho).
\end{equation}

We use the
symbol notation
\begin{align*}
& \sigma(\partial_{\rho})
=i\eta\\
&  \sigma(\partial_{x_j})
= i\xi_{j} \qquad j=1,\ldots, 2n-1.
\end{align*}
The second order
terms of the
(diagional matrix) operator
in $\square$
 from Proposition \ref{squareOp} are given by
\begin{equation*}
 -\sum_{l\in J} \Lb_l L_l 
- \sum_{l\notin J}  L_l\Lb_l 
.
\end{equation*}
Expanding this operator using
 \eqref{Lj} and \eqref{LnRepeat}, we write
in local coordinates
\begin{equation*}
-\sum_{l\in J} \Lb_l L_l 
- \sum_{l\notin J}  L_l\Lb_l 
 = 
   - \frac{1}{2}\Bigg(
 \frac{\partial^2}{\partial\rho^2}+
 \frac{1}{2} \sum_j \frac{\partial ^2}{\partial x_j^2}
 +
 2 \frac{\partial ^2}{\partial x_{2n-1}^2}
 +  \sum_{j,k=1}^{2n-1} l_{jk}
 \frac{\partial ^2}
 {\partial x_j\partial x_k}
 \Bigg) + O(\rho)
 ,
\end{equation*}
 where $l_{jk}=O(x)$, and
 modulo first order terms.  We
  define the operator $\Gamma$ to be given by
 the terms without a $\rho$ factor on the right
  hand side:
\begin{equation*}
\Gamma :=  - \Bigg(
\frac{\partial^2}{\partial\rho^2}+
\frac{1}{2} \sum_j \frac{\partial ^2}{\partial x_j^2}
+
2 \frac{\partial ^2}{\partial x_{2n-1}^2}
+  \sum_{j,k=1}^{2n-1} l_{jk}
\frac{\partial ^2}
{\partial x_j\partial x_k}
\Bigg).
\end{equation*}

We let $L_{bj}$ for 
 $j=1,\ldots, n-1$ be defined by
\begin{equation*}
 \sigma(L_{bj}) = \left. 
  \sigma(L_{j}) \right|_{\rho=0}.
\end{equation*}
Then we have
\begin{align*}
\sigma_2(\square) =&
\frac{1}{2} \eta^2+\xi_{2n-1}^2
+ \sum_{j=1}^{n-1} \sigma(L_{bj})
\sigma(\Lb_{bj}) + O(\rho)O(\xi^2)\\
=&
\frac{1}{2} \eta^2+\xi_{2n-1}^2
+ \frac{1}{4}\sum_{j=1}^{2n-2} \xi_{j}^2
+O(x)O(\xi^2)+O(\rho)O(\xi^2)\\
=& \frac{1}{2} \sigma(\Gamma)+O(\rho)O(\xi^2)
.
\end{align*}
We use the Kohn-Nirenberg notation,
$\sigma_j$ to denote the 
part of a symbol
homogeneous of degree j in
$\xi$ and $\eta$ in its symbol 
expansion.

Let us now denote
\begin{align}
\nonumber
\Xi^2(x,\xi) =& 2\xi_{2n-1}^2
+ 2\sum_{j=1}^{n-1} \sigma(L_{bj})
\sigma(\Lb_{bj})\\
\label{xiLb}
=& 2\xi_{2n-1}^2
+ \frac{1}{2}\sum_{j=1}^{2n-2} \xi_{j}^2
+O(x)O(\xi^2)
\end{align}
so that we can write
\begin{equation*}
  \sigma(\Gamma) =
\eta^2 + \Xi^2(x,\xi).
\end{equation*}

We now collect the second order $O(\rho)$ terms 
from $2\square$ 
 in an operator, $\tau$, i.e.
\begin{equation*}
\sigma(\tau)\Big|_{\rho=0}
=2 \frac{\partial}{\partial\rho}
\sigma_2(\square) \Big|_{\rho=0},
\end{equation*}
and all tangential first order operators in the expression
of the operator $2\square$ as in Proposition 
\ref{squareOp} into a pseudodifferential operator, denoted
$A$.  We also denote by the operator $S$ the 
zero order operator which is multiplication by the 
(matrix) coefficient of $\frac{\partial}{\partial\rho}$
in the operator $2\square$.

With the notation 
$
v|_{\rho=0} = 
g_b(x) ,
$
the equation $2\square v=0$ can be written
locally as
\begin{equation}
\label{DirProbCorrOps}
\Gamma v
+ \sqrt{2}S \left(\frac{\partial v}{\partial\rho}\right)
+ A v+\rho\tau(v)
=0.
\end{equation}

The Dirichlet to Neumann operator
(DNO) is
defined here as the boundary 
operator producing the boundary values of the
outward normal derivative 
of the solution to 
the Poisson equation
$2\square v =0$, with boundary
 values $v=g_b$ on 
$\partial\Omega$.  In the equation
\eqref{DirProbCorrOps} above, the 
DNO can be found by solving for
$\partial_{\rho}v\big|_{\rho=0}$.

We rewrite \eqref{DirProbCorrOps} using Fourier Transforms,
 extending \eqref{DirProbCorrOps} to 
$\mathbb{R}^{2n}$ by 0.  Let $E$ denote the
 extension by 0.
The term $E\circ \Gamma v$ can be written
\begin{equation*}
E\circ \Gamma v = \Gamma\circ E  v -
\frac{1}{(2\pi)^{2n}}
\int\left( \partial_{\rho}\widetilde{v}
\Big|_{\rho=0}
+i\eta \widetilde{g}_b(\xi)
\right)
e^{i\rho\eta} e^{ix\cdot \xi}
d\xi d\eta.
\end{equation*}
For ease of notation, we will disregard
 the extension operator, $E$, and instead
use the subscript $int$ to signify an operator 
 is to be applied to the extension by 0 to 
$\mathbb{R}^{2n}$ of a distribution defined in 
 $\Omega$.  With this convention, we write
\begin{equation*}
\Gamma v = \Gamma_{int} v -
\frac{1}{(2\pi)^{2n}}
 \int\left( \partial_{\rho}\widetilde{v}
\Big|_{\rho=0}
+i\eta \widetilde{g}_b(\xi)
\right)
e^{i\rho\eta} e^{ix\cdot \xi}
d\xi d\eta,
\end{equation*}
 where $\Gamma$ on the left-hand side
  is to be understood
as an operator $\Gamma : \mathscr{D}'(\Omega)
 \rightarrow  \mathscr{D}'(\mathbb{R}^{2n})$
via (left-side) composition with $E$,  
 and 
where
 $\Gamma_{int}:\mathscr{D}'(\mathbb{R}^{2n}) 
 \rightarrow \mathscr{D}'(\mathbb{R}^{2n})$
has as symbol:
\begin{equation*}
\nonumber
\sigma\left(
\Gamma_{int}
\right) 
= \eta^2 + \Xi^2(x,\xi).
\end{equation*}

The term $S
\left(\frac{\partial v}{\partial\rho}\right)$ can be written
\begin{align*}
S
\left(\frac{\partial v}{\partial\rho}\right)
&=\frac{1}{(2\pi)^{2n}}
 \int s(x,\rho)
  i\eta \widehat{v}(\xi,\eta)
e^{i\rho\eta} e^{ix\cdot \xi}
d\xi d\eta + 
 \frac{1}{(2\pi)^{2n}}\int s(x,0) \widetilde{g}_b(\xi)
e^{ix\cdot \xi}    d\xi \\
&:= S_{int}v + S_b g_b,
\end{align*}
 where similarly the left-hand side is
understood to be composed on the left by
 $E$, and where $S_{int}:=S\circ E$.  We have 
 $\sigma(S_{int})
  = s(x,\rho) i \eta$, and
 $S_b \in \Psi^0(\partial\Omega)$ with
$\sigma(S_b) = s(x,0)$.
 From Proposition \ref{squareOp}, 
$s(x,\rho)$ is a diagonal matrix (of 
 smooth functions).

We now rewrite 
 \eqref{DirProbCorrOps}
  as
\begin{multline*}
\nonumber
\Gamma_{int}v = 
\frac{1}{(2\pi)^{2n}} 
 \int\left( \partial_{\rho}\widetilde{v}
\Big|_{\rho=0}
+i\eta \widetilde{g}_b(\xi)
\right)
e^{i\rho\eta} e^{ix\cdot \xi}
d\xi d\eta\\
- \sqrt{2} S_{int} v
- \sqrt{2}  S_{b} g_b
-A v
- \rho  \tau v,
\end{multline*}
where $v$ is understood to be extended by 0 to
 all of $\mathbb{R}^{2n}$.

  $\Gamma_{int}$ is an elliptic operator
on $\mathbb{R}^{2n}$ and so 
we can apply an
inverse to $\Gamma_{int}$:
\begin{align}
\nonumber
v=&
\frac{1}{(2\pi)^{2n}}
 \Gamma^{-1}_{int} \circ \int\left(
   \partial_{\rho}\widetilde{v}
\Big|_{\rho=0}
+i\eta \widetilde{g}_b(\xi)
\right) 
e^{i\rho\eta} e^{ix\cdot \xi}
 d\xi d\eta\\
&- \sqrt{2}\Gamma^{-1}_{int} \circ
S_{int} v
- \sqrt{2}\Gamma^{-1}_{int} \circ
S_{b} g_b
\label{vSoln}
 - \Gamma^{-1}_{int} \circ  A v
-\Gamma^{-1}_{int} \circ ( \rho
\tau v ) 
\end{align}
modulo smoothing terms.
The idea behind our calculations of the DNO is to
write $ \partial_{\rho}v
\Big|_{\rho=0} = \Lambda_b^1g_b + 
 \Lambda^0_b g_b +\cdots$, 
insert this expansion into 
the first integral on the right-hand side of
\eqref{vSoln},
set $\rho=0$ in \eqref{vSoln}, and
equate terms with the same order, or, equivalently,
of the same degree 
in $\Xi(x,\xi)$
(see \cite{CNS92} for another approach).

 We first prove a Proposition
  about the Poisson operator, giving the
 solution, $v$, above.
   In order to consolidate the various smoothing
 terms which arise, we write $R_b^{-\infty}$ to include the
 restriction to $\rho=0$ of any sum of smoothing operators in
 $\Psi^{-\infty}(\Omega)$ applied to $v$,
 or smoothing operators in $\Psi^{-\infty}_b(\partial\Omega)$ applied to
 the boundary values
 $g_b$ or 
 $ \partial_{\rho} v |_{\rho=0}$.
    We also write $R^{-\infty}$ to include
 any sum of smoothing operators in
 $\Psi^{-\infty}(\Omega)$ applied to $v$,
 smoothing operators in $\Psi^{-\infty}(\Omega)$ applied to
 $g_b\times\delta(\rho)$ or 
 $ \partial_{\rho}
   v|_{\rho=0}\times\delta(\rho)$,
 or decomposable operators in $\Psi^{-k}(\Omega)$ for $k \ge 1$
 applied to $R_b^{-\infty}$
 (such terms can thus be estimated in terms of smooth boundary terms, 
 see Theorem \ref{estinvell})
 as well as smoothing operators in $\Psi^{-\infty}(\Omega)$ applied to
 $R_b^{-\infty}$.  From the definitions
 we have $R\left( R^{-\infty} \right) =
 R_b^{-\infty}$.
 
Estimates for the Poisson operator corresponding
 to an elliptic operator were worked out 
in \cite{Eh18_halfPlanes}.  In those results, the
 highest order term of the DNO was also 
  calculated.  The calculations here follow those
in \cite{Eh18_halfPlanes} to find the Poisson operator
 corresponding to $\square$.  As the operator, $\square$,
  is slightly different than the operator considered in
the author's earlier work
  (namely in the first order terms), 
and as the Poisson operator will be 
 used to obtain the lower order terms of the DNO,
we go through the calculations in detail, 
 obtaining first an expression for the 
Poisson operator, and then calculating the 
 DNO.

 We define the Poisson operator corresponding to $\square$ as 
the
 solution operator,
 $P$, mapping $(0,q)$-forms on $\partial \Omega$
 to $(0,q)$-forms on $\Omega$, to
 \begin{equation}
 \label{PoissEqn}
\begin{aligned} &2\square\circ P = 0\\
&R\circ P = I.
\end{aligned}
 \end{equation}
 We assume the classical results guaranteeing existence
  and uniqueness of a solution.

 \begin{thrm}
 \label{poiss}
 	Let $g_b$ be a $(0,q)$-form on $\partial \Omega$; each
 	component of $g_b$ is a distribution supported on
 	$\partial\Omega$.  Let $g=g_b(x) \times \delta(\rho)$ in local coordinates.
 	Then
\begin{equation*}  
 	P g = \Psi^{-1} g + R^{-\infty}.
\end{equation*}
 \end{thrm}
 \begin{proof}
 	From \eqref{vSoln}, we have
 modulo $R^{-\infty}$
\begin{align}
 	\nonumber
 	v
 	=& \frac{1}{(2\pi)^{2n}}\int 
 	\frac{
 		\partial_{\rho} \widetilde{v} (\xi,0)
 		+ i\eta \widetilde{g}_b(\xi)
 	}{\eta^2+\Xi^2(x,\xi)}
 	e^{i x\xi}e^{i\rho\eta} d\xi d\eta\\
 	\nonumber
 	& - \sqrt{2}\Gamma^{-1}_{int} \circ S\left( \frac{\partial
 		v}{\partial\rho}
 	\right)
 	-  \Gamma^{-1}_{int} \circ A( v) -
 	\Gamma^{-1}_{int}
 	  \circ  \rho \tau ( v)\\
 	&  + \Psi^{-3} 
 \left(\left.\partial_{\rho} v\right|_{\rho=0}
 	\times \delta(\rho) \right)
 	+\Psi^{-2} g
 	+ \Psi^{-2}v,
 	\label{ftrepeat}
\end{align}
locally, in a small neighborhood of the origin; 
we recall, using the pseudodifferential 
 analysis, we consider $v$ to have compact 
support in a neighborhood of a boundary point,
 which we take to be the origin, and
  the pseudodifferential operators are also
 composed on the left with cutoffs with support
  in a neighborhood of the origin; see the
discussion in Section \ref{notation} as 
 well as the discussion following
\eqref{2bgl}.  For ease of notation, we omit
 the writing of the cutoffs. 
 We will also omit mention of the
smooth $R^{-\infty}$ terms, inserting them
 again at the end of the calculations.
 
 We note the terms
 	$\Gamma^{-1}_{int} 
  \circ S\left( \frac{\partial
 		v}{\partial\rho}
 	\right)$ and
 	$\Gamma^{-1}_{int} \circ  A( v)$
 	contribute terms
 	$\Psi^{-1}  v$ and $\Psi^{-2}  g$.
 	
 To handle the term
 \begin{equation}
 	\label{ligterm}
 	\Gamma^{-1}_{int}  \circ  \rho \tau ( v),
 \end{equation}
we write the operator $\tau$ using the
 form of its symbol
\begin{equation*}
 \sigma (\tau) = 
   \sum_{j,k=1}^{2n-1} \tau^{jk}(x,\rho)
    \xi_j\xi_k,
\end{equation*}
and	we rearrange \eqref{ftrepeat} as
\begin{align*}
 	\nonumber
 	v
 	=& \frac{1}{(2\pi)^{2n}}\int
  \frac{\partial_{\rho} \widetilde{v}
  (\xi,0) + i\eta \widetilde{ g}_b(\xi)}
 	{\eta^2+\Xi^2(x,\xi)}
 	e^{i x\xi}e^{i\rho\eta} d\xi d\eta\\
 	&+ \frac{1}{(2\pi)^{2n}}
\int \rho \frac{\sum_{j,k}
 		\tau^{jk}(x,\rho) \xi_j \xi_k}
 	{\eta^2 + \Xi^2(x,\xi)}
 	\widehat{ v}(\xi,\eta)e^{i x\xi}e^{i\rho\eta} d\xi d\eta
 	\\
 	&  + \Psi^{-3} \left(\left.
 \partial_{\rho} v\right|_{\rho=0}
 	\times \delta(\rho) \right)
 	+\Psi^{-2} g
 	+ \Psi^{-1}v,
\end{align*}
 	as the terms involving the operators $S$
 and $A$ are included in
 	the last two remainder terms.   We then bring the second term on
 	the right to the left-hand side:
\begin{align}
\nonumber
\frac{1}{(2\pi)^{2n}}\int
 	\Bigg(1-&\rho\frac{\sum_{j,k} \tau^{jk}(x,\rho) \xi_j \xi_k}
 	{\eta^2 + \Xi^2(x,\xi)}\Bigg) \widehat{v}(\xi,\eta)
 	e^{i x\xi}e^{i\rho\eta} d\xi d\eta=\\
 	\nonumber
 	&\frac{1}{(2\pi)^{2n}}\int
 	 \frac{
 		\partial_{\rho} \widetilde{v} (\xi,0)
 		+ i\eta \widetilde{g}_b(\xi)
 	}{\eta^2+\Xi^2(x,\xi)}
 	e^{i x\xi}e^{i\rho\eta} d\xi d\eta\\
 	\label{tauleft}
 	&
 	+ \Psi^{-3} \left(\left.
 \partial_{\rho} v\right|_{\rho=0}
 	\times \delta(\rho) \right)
 	+\Psi^{-2} g
 	+ \Psi^{-1}v.
\end{align}
For small enough $\rho$ 
 (which, without loss of generality,
  can be assumed by choosing the cutoffs defining 
the pseudodifferential operators 
 appropriately small) the symbol
\begin{equation}
 	\label{oprho1}
 	1-\rho\frac{\sum_{j,k=1}^{2n-1}
  \tau^{jk}(x,\rho) \xi_j \xi_k}
{\eta^2	+ \Xi^2(x,\xi)}
 	\end{equation}
 	is non-zero, and so (shrinking the
 	 support of $v$ if necessary) 
 we can apply a parametrix of the operator with
 	symbol \eqref{oprho1} to both sides of \eqref{tauleft}.  We note
 	the symbol of such an operator is of the form $1+O(\rho)$, where
 	the second term is a symbol of class $\mathcal{S}^0(\Omega)$,
 	which is $O(\rho)$.
We obtain
\begin{align}
\nonumber
v=&
 	\frac{1}{(2\pi)^{2n}}\int
 	\left(1+O(\rho)\right)
\frac{ \partial_{\rho} \widetilde{v}
  (\xi,0) +i\eta \widetilde{g}_b(\xi)}
 	{\eta^2 + \Xi^2(x,\xi)}e^{i x\xi}e^{i\rho\eta} d\xi d\eta\\
 	\label{oprho}
 	& + \Psi^{-3} \left(\left.
\partial_{\rho} v\right|_{\rho=0}
 	\times \delta(\rho) \right)
 	+\Psi^{-2} g
 	+ \Psi^{-1}v .
 	\end{align}
 	
From Lemma \ref{liglem} we have that
\begin{align*}
 	\nonumber \frac{1}{(2\pi)^{2n}}
 \int  O(\rho) &\frac{
 	\partial_{\rho} \widetilde{v} (\xi,0)
 +i\eta \widetilde{g}_b(\xi)}
 	{\eta^2 + \Xi^2(x,\xi)}
 	e^{i x\xi}e^{i\rho\eta} d\xi d\eta\\
 	& = \Psi^{-3} \left(\left.
\partial_{\rho} v \right|_{\rho=0}
 	\times \delta(\rho) \right)
 	+\Psi^{-2} g.
\end{align*}
 	
Returning to \eqref{oprho} we write
\begin{align}
 	\nonumber
v= \frac{1}{(2\pi)^{2n}} \int  
 	\frac{\partial_{\rho} \widetilde{v}(\xi,0)
+i\eta \widetilde{g}_b(\xi)}
 	{\eta^2 + \Xi^2(x,\xi)}&e^{i x\xi}e^{i\rho\eta} d\xi d\eta\\
 	\label{vdvg}
 	& + \Psi^{-3} \left(\left.
\partial_{\rho} v \right|_{\rho=0}
 	\times \delta(\rho) \right)
 	+\Psi^{-2} g
 	+ \Psi^{-1}v.
\end{align}
 The expression above is locally confined to
  a neighborhood of the origin,
 but using coverings
   and a partition of unity (as in the 
 explanation in \eqref{alld}) we can obtain an 
expression for $v$ on all of $\Omega$.  Then 
 inverting an operator of the form
 	$I-\Psi^{-1}$
 	gives an expression for $v$ on $\Omega$:
\begin{equation}
 	\label{subv}
 	v
 	= \Psi^{-2} \left(\left.
\partial_{\rho} v \right|_{\rho=0}
 	\times \delta(\rho) \right) +
 	\Psi^{-1}g + \Psi^{-\infty} v,
\end{equation}
 	as a vector-valued relation, with matrix-valued pseudodifferential
 	operators.

Using the residue calculus, we can take an
  inverse transform in
 	\eqref{vdvg}
 	with respect to $\eta$.  For $\rho\rightarrow 0^+$, we have
\begin{align*}
0 =& \frac{1}{(2\pi)^{2n}} 2\pi i \left(
 	\int \frac{1}{2i|\Xi(x,\xi)|}
 	\partial_{\rho} \widetilde{v} (\xi,0)
 e^{ix\xi}d\xi
 	- \int \frac{|\Xi(x,\xi)|}{2i|\Xi(x,\xi)|}
 	\widetilde{g}_b(\xi)
 	e^{ix\xi}d\xi \right)\\
 	&
 	+ R\circ 
 \Psi^{-3} \left(\left.
 \partial_{\rho} v\right|_{\rho=0}
 	\times \delta(\rho) \right)
 	+R\circ \Psi^{-2} g
 	+ R\circ \Psi^{-1}v\\
 	=&
 	\frac{1}{(2\pi)^{2n-1}} \int \left(
 	\frac{1}{2|\Xi(x,\xi)|}
 	\partial_{\rho} \widetilde{v} (\xi,0)
 	-\frac{1}{2}\widetilde{g}_b(\xi) \right)e^{ix\xi}d\xi\\
 	&
 	+ \Psi^{-2}_b\left( \left.
\partial_{\rho} v\right|_{\rho=0} \right)
 	+ \Psi^{-1}_b g_b
 	+ R \circ \Psi^{-1} v,
\end{align*}
where we apply Lemma \ref{restrict} to the
 	terms with operators 
$R\circ \Psi^{-3}$ and 
$R\circ \Psi^{-2}$ in the second step.
We can now invert the operator with
 symbol $1/2|\Xi(x,\xi)|$ and 
solve for $\partial_{\rho} v |_{\rho=0}$:
\begin{equation}
 	\label{1dno}
 	\frac{\partial v}{\partial \rho}
 	(x,0) = \int |\Xi(x,\xi)| 
 	\widetilde{ g}_b(\xi)e^{ix\xi}d\xi
 	+\Psi_b^0g_b+
\Psi^1_b\circ R\circ \Psi^{-1} v,
\end{equation}
 locally, in a small neighborhood of 
the origin.
 	Alternatively, we could in a similar manner use the residue
 	calculus to take an inverse transform with respect to $\eta$ in
 	\eqref{vdvg} and calculate for $\rho\rightarrow 0^-$ with the same
 	result.
 	
For the term 
$\Psi_b^1 \circ R \circ \Psi^{-1} v$
 we insert \eqref{subv} in the argument:
\begin{align*}
\Psi_b^1 \circ R\circ  \Psi^{-1} v
 	=& \Psi_b^1 \circ R \circ  \Psi^{-1}
\left(\Psi^{-2} 
\left(\left.\partial_{\rho} v\right|_{\rho=0}
 	\times \delta(\rho) \right) +
 	\Psi^{-1}g + \Psi^{-\infty} v\right)\\
=&\Psi_b^1 \circ R \circ \Psi^{-3}
 \left(\left.\partial_{\rho} v\right|_{\rho=0}
 	\times \delta(\rho) \right) +
 \Psi_b^1 \circ R \circ \Psi^{-2}g
  +\Psi_b^1 \circ R\circ \Psi^{-\infty} v\\
 	=&\Psi_b^1 \circ \Psi_b^{-2}
 \left(\left.\partial_{\rho} v\right|_{\rho=0}
 	\right) +
 	\Psi_b^1 \circ \Psi_b^{-1}g_b
 +R\circ \Psi^{-\infty} v\\
 	=& \Psi_b^{-1} \left(\left.
\partial_{\rho} v\right|_{\rho=0}
 	\right) +
 	\Psi_b^{0}g_b + R_b^{-\infty} v .
 	\end{align*}
 \eqref{1dno} above leads to
 	the well-known result that the DNO is a first order operator on
 	the boundary data,
 	with principal term $|\Xi(x,\xi)|$:
\begin{equation*}
 	\left.\frac{\partial v}
 	{\partial \rho}\right|_{\rho=0} =
 	\int |\Xi(x,\xi)| 
 \widetilde{g}_b(\xi)e^{ix\xi}d\xi
 	+\Psi_b^0 g_b+
 	\Psi_b^{-1} \left(\left.
\partial_{\rho} v\right|_{\rho=0} \right)
 	+
 	R_b^{-\infty} v.
\end{equation*}
 	
 Again, using a covering and the local 
  expressions to obtain a global relation, 
 and  solving for
 	(the vector) 
 $\left.\partial_{\rho} v \right|_{\rho=0} $,
  and absorbing extra
$\Psi^{-\infty}_b \left(\left.
\partial_{\rho} v\right|_{\rho=0}
 	\right)$
terms into the remainder term, $R_b^{-\infty}$,
 	leads to the expression:
\begin{equation}
 \label{1der}
\left.\partial_{\rho} v\right|_{\partial\Omega}
 	= |D| g_b + \Psi_b^0 g_b + R^{-\infty}_b
\end{equation}
 	where $|D|$ is defined as the first order operator with symbol
 	locally given by $\sigma(|D|)=|\Xi(x,\xi)|$.
 	
 We can now insert \eqref{1der} in \eqref{vdvg}
   and obtain in a small neighborhood of
the origin
\begin{align*}
\nonumber
v=& \frac{1}{(2\pi)^{2n}}
 \int  \frac{
 		(|\Xi(x,\xi)| +i\eta )
 \widetilde{g}_b(\xi)}
 {\eta^2 + \Xi^2(x,\xi)}
  e^{i x\xi}e^{i\rho\eta} d\xi d\eta\\
& + \Psi^{-3} \left(\left.
\partial_{\rho} v \right|_{\rho=0}
 	\times \delta(\rho) \right)
 	+\Psi^{-2} g
 	+ \Psi^{-1}v +\Psi^{-2}R_b^{-\infty}\\
 	=& \frac{i}{(2\pi)^{2n}}
\int 
 	\frac{\widetilde{g}_b(\xi)}
 	{\eta + i|\Xi(x,\xi)|}e^{i x\xi}e^{i\rho\eta} d\xi d\eta\\
 	& + \Psi^{-3} \left((\Psi_b^1g_b+R_b^{-\infty})\times\delta(\rho) \right)
 	+\Psi^{-2} g
 	+ \Psi^{-1}v  +R^{-\infty},
\end{align*}
 	which we write as
 	\begin{equation}
 	\label{locpoiss} v = \Psi^{-1}g + \Psi^{-1}v+R^{-\infty}.
 	\end{equation}
 	We thus obtain
 	\begin{equation*}
 	v= \Psi^{-1} g +R^{-\infty},
 	\end{equation*}
 on all of $\Omega$.
 \end{proof}
 From the proof of the Theorem we also have the principal symbol of
 the operator $\Psi^{-1}$ acting on $g_b\times
 \delta(\rho)$; it is given locally by
 (the diagonal matrix)
 \begin{equation}
 \label{symlam}
   \frac{i}
 {\eta + i|\Xi(x,\xi)|},
 \end{equation}
 which we note for future reference.  Using the representation as
 in \eqref{1der},
\begin{equation}
 \label{1der2}
 \left.\partial_{\rho} v\right|_{\partial\Omega}
 = \Psi_b^1 g_b + R^{-\infty}_b,
\end{equation}
 we can obtain with Lemma \ref{estinvell}
 estimates for the Poission operator.
 
We first handle the smooth terms, 
 $R^{-\infty}$ and $R_b^{-\infty}$:
\begin{lemma}
\label{lemRemainderEst}
 For  $R^{-\infty}$ and $R_b^{-\infty}$,
 and $g_b$,
 defined as above, we have
  for all $s$
\begin{equation*}
\| R^{-\infty} \|_{W^{s}(\Omega)} 
\lesssim  \|g_b\|_{L^2(\partial\Omega)}
\end{equation*}
and
\begin{equation*}
\| R^{-\infty}_b \|_{W^{s}(\partial\Omega)} 
\lesssim  \|g_b\|_{L^2(\partial\Omega)}.
\end{equation*}
\end{lemma}
\begin{proof}
We note the
$L^2$ estimates for the Poisson operator,
$\|P(g)\|_{L^2(\Omega)} \lesssim
\|g_b\|_{L^2(\partial\Omega)}$
(see for instance \cite{LiMa}).  	
	
 For $R^{-\infty}$, we have by definition
\begin{align*}
\nonumber
\| R^{-\infty} \|_{W^{s}(\Omega)} 
\lesssim& \|u\|_{W^{-\infty}(\Omega)}
+\|g_{b}\|_{W^{-\infty}
	(\partial \Omega)}
+ \|\partial_{\rho}u
\big|_{\partial\Omega}
\|_{W^{-\infty}
	(\partial \Omega)}\\
\nonumber
\lesssim& 
 \|g_{b}\|_{L^2(\partial \Omega)} +
\|\partial_{\rho}u \big|_{\partial\Omega} \|_{W^{-\infty}
	(\partial \Omega)}
\end{align*}
for any $s\ge 0$.

We can estimate boundary values of 
a term, $\partial_{\rho}u
\big|_{\partial\Omega_j}$ by 
assuming support in a neighborhood of
$\partial\Omega$ intersected with $\Omega$
and writing
\begin{align*}
\partial_{\rho}u
\big|_{\rho =0 } 
=& \int_{-\infty}^0 \partial_{\rho}^2 u d\rho \\
=& \int_{-\infty}^0 D_{t}^2 u d\rho + 
 \int_{-\infty}^0 \left( \phi_1 \partial_{\rho}u + \phi_2 u\right) 
  d\rho,
\end{align*}
where $D_{t}^2$ is a second order tangential 
 operator, and $\phi_1$ and $\phi_2$
  are smooth with support in the interior of $\Omega$.
From interior regularity, we have
\begin{equation*}
 \| \phi_j u \|_{W^2(\Omega)} \lesssim \|g_b\|_{L^2(\partial\Omega)}.
\end{equation*}
Thus,
applying a tangential smoothing operator to
both sides and integrating yields
\begin{align*}
\|\partial_{\rho}u
\big|_{\partial\Omega}
\|_{W^{-\infty}
	(\partial \Omega)}
\lesssim \|g_b\|_{L^2(\partial\Omega)}.
\end{align*}

Hence,
\begin{equation*}
\| R^{-\infty} \|_{W^{s}(\Omega)} 
\lesssim  \|g_b\|_{L^2(\partial\Omega)}.
\end{equation*}
 The estimates for 
  $R_b^{-\infty}$ follow similarly.
\end{proof} 
 
 \begin{thrm}
 	\label{poissest}
Let $P$ be the Poisson operator on $\Omega$ for
the system \eqref{PoissEqn}.  Then
 for $s\ge 0$
 	\begin{equation*}
 	\| P (g) \|_{W^{s+1/2}(\Omega)} \lesssim
 	\|g_b\|_{W^{s}(\partial \Omega)}.
 	\end{equation*}
 \end{thrm}
\begin{proof} 	
We use the representation 
 $P(g) = \Psi^{-1} g +R^{-\infty}$
as in Theorem \ref{poiss},
  where the $\Psi^{-1}$ operator is
decomposable.  The estimates then follow from
Theorem \ref{estinvell} and
Lemma \ref{lemRemainderEst}.
 \end{proof}
 
 Included in the proof of Theorem \ref{poiss} is the
 calculation of the highest order term of the DNO;
 from \eqref{1der} we have in particular the first component of
 the DNO, which we write as, $N^-$
(the $-$ superscript to denote we compute 
 the outward pointing normal derivative):
 \begin{thrm}
 	\label{thrmdno1}
 	\begin{equation}
 	\label{dnoprincipal}
 	N^- g = |D| g_b+
 	\Psi^0_b( g_b) + R^{-\infty}_b.
 	\end{equation}
 \end{thrm}

 We now want to write out the highest order terms included in
 $\Psi^0_b ( g)$ in \eqref{dnoprincipal}.
 That is to say,
writing $\partial_{\rho}v
\Big|_{\rho=0} = \Lambda_b^1g_b + 
\Lambda^0_b g_b +\cdots$, we have
 $\Lambda_b^1 = |D|$, and we want to 
calculate an expression for the operator
 $\Lambda_b^0$. 
 
Recall in \eqref{vSoln} we had the relation
\begin{align*}
\nonumber
v=&
\frac{1}{(2\pi)^{2n}}
\Gamma^{-1}_{int} \circ \int\left(
  \partial_{\rho}\widetilde{v}
\Big|_{\rho=0}
+i\eta \widetilde{g}_b(\xi)
\right) 
e^{i\rho\eta} e^{ix\cdot \xi}
d\xi d\eta\\
&- \sqrt{2}\Gamma^{-1}_{int} \circ
S_{int} v
- \sqrt{2}\Gamma^{-1}_{int} \circ
S_{b} g_b
- \Gamma^{-1}_{int} \circ  A v
-\Gamma^{-1}_{int} \circ ( \rho
\tau v ) 
\end{align*}
modulo smooth terms.  
 With $\partial_{\rho}v
 \Big|_{\rho=0} = \Lambda_b^1g_b + 
 \Lambda^0_b g_b +\cdots$, 
and using
\begin{equation*}
\frac{1}{(2\pi)^{2n}}
  \Gamma^{-1}_{int} 
  \circ \int\left( \partial_{\rho}
  \widetilde{v}
  \Big|_{\rho=0}
  +i\eta \widetilde{g}_b(\xi)
  \right) e^{i\rho\eta} e^{ix\cdot \xi}
   d\xi d\eta
= \Theta^+ g + \Psi^{-2}g + R^{-\infty},
\end{equation*} 
we can write the relation as
\begin{align}
\nonumber
v=&
 \Theta^+ g
+ \Psi^{-2}g
+ \Gamma^{-1}_{int} \circ \Lambda^0_b g
- \sqrt{2}\Gamma^{-1}_{int} \circ
S_{int} v
- \sqrt{2}\Gamma^{-1}_{int} \circ
S_{b} g_b\\
 \label{homogE}
&  - \Gamma^{-1}_{int} \circ  A v
-\Gamma^{-1}_{int} \circ ( \rho
\tau v )+ R^{-\infty},
\end{align}
where $\Theta^+$ is defined by
\begin{equation*}
\sigma(\Theta^+)
=i \frac{1}
{\eta + i|\Xi(x,\xi)|}.
\end{equation*}
 The pseudodifferential calculus 
also yields the principal term
 of the  $\Psi^{-2}$ operator in
\eqref{homogE}.   The
 operator arises in the expansion of the
  symbol for the inverse,
$\Gamma^{-1}_{int}$:
\begin{equation*}
 \sigma\left( \Gamma^{-1}_{int} \right)
  = \frac{1}{\eta^2+\Xi^2(x,\xi)}
   + \frac{\partial_{\xi}\Xi^2 \cdot D_x \Xi^2}{(\eta^2+\Xi^2(x,\xi))^3}
   +\cdots.
\end{equation*}
And so the principal symbol of the $\Psi^{-2}$
operator in \eqref{homogE} is given by
\begin{equation}
 \label{psiLowSymbol}
 \frac{\partial_{\xi}\Xi^2 \cdot D_x \Xi^2}{(\eta^2+\Xi^2(x,\xi))^3}
  (|\Xi(x,\xi)|+i\eta)
   = \frac{ \partial_{\xi}\Xi^2 \cdot \partial_x 
   	\Xi^2}{(\eta^2+\Xi^2(x,\xi))^2
   (\eta+i|\Xi(x,\xi)|)}
    .
\end{equation}

For the term $\Lambda^0_b$,
we set $\rho=0$
in \eqref{homogE} and look at the terms of order
$-1$ in $\Xi(x,\xi)$.
 The first term, $\Theta^+$ leads to a 
  term which is homogeneous of 
   order $0$ in $|\Xi(x,\xi)|$.
We go through the other terms individually.
For the operator with symbol as in 
 \eqref{psiLowSymbol} we calculate
\begin{align*}
 \frac{1}{(2\pi)^{2n}}
 \int  \frac{ \partial_{\xi}\Xi^2 \cdot \partial_x 
 	\Xi^2}{(\eta^2+\Xi^2(x,\xi))^2
 	(\eta+i|\Xi(x,\xi)|)}& \widetilde{g}_b(\xi)
  e^{i\xi x}  d\eta d\xi\\
   =& -\frac{3i}{(2\pi)^{2n-1}}
   \int  \frac{ \partial_{\xi}\Xi^2 \cdot \partial_x 
   	\Xi^2}{16 \Xi^4(x,\xi)}& 
    \widetilde{g}_b(\xi)
   e^{i\xi x}   d\xi.
\end{align*}

Next, we have
\begin{equation*}
R\circ 
 \Gamma_{int}^{-1} \circ\Lambda_b^0 g=
\frac{1}{(2\pi)^{2n-1}}
\int \frac{\widetilde{\Lambda_b^0 g_b}
	(\xi)}
{2|\Xi(x,\xi)|}e^{i\xi x}  d\xi
+ \Psi_b^{-2} g.
\end{equation*}
For terms involving $v$ we use the expression
\begin{equation}
\label{PoissonSoln}
v= \Theta^+ g + \Psi^{-2}g
\end{equation}
modulo smoothing terms
as in Theorem \ref{poiss}.

With \eqref{PoissonSoln}, 
 and $s_0(x):= s(x,0)$, we thus have
\begin{align*}
R\circ  \Gamma^{-1}_{int} \circ
S_{int} v = &
- \frac{1}{(2\pi)^{2n}} \int s_0(x) \frac{\eta}
{\eta^2 + \Xi^2(x,\xi)}
\frac{\widetilde{g}_b(\xi)}
{\eta +i |\Xi(x,\xi)|} d\eta
e^{ix\cdot\xi}
d\xi\\
=& - \frac{1}{(2\pi)^{2n-1}}
\int s_0(x)
\frac{\widetilde{g}_b(\xi)}
{4|\Xi(x,\xi)|}
e^{ix\cdot\xi}   d\xi,
\end{align*}
modulo lower order terms.
 Note that any $O(\rho)$ terms from an expansion
  of $s(x,\rho)=s_0(x) + O(\rho)$ lead to
lower order terms by Lemma \ref{liglem}.

Next,
\begin{align*}
R\circ
\Gamma^{-1}_{int} \circ
S_{b} g_b = &
\frac{1}{(2\pi)^{2n}} \int s_0(x)
\frac{\widetilde{g}_b(\xi)}
{\eta^2 + \Xi^2(x,\xi)}
d\eta
e^{ix\cdot\xi} d\xi\\
=&  \frac{1}{(2\pi)^{2n-1}}
\int s_0(x)
\frac{\widetilde{g}_b(\xi)}
{2|\Xi(x,\xi)|}
e^{ix\cdot\xi}  d\xi,
\end{align*}
modulo lower order terms.

Similar to the calculation involving
$\Gamma^{-1}_{int} \circ
S_{int} v$ above, we have
for $\Gamma^{-1}_{int} \circ
A v$
\begin{equation*}
R\circ \Gamma^{-1}_{int} \circ
A v =   \frac{1}{(2\pi)^{2n-1}}
\int a_{0}(x,\xi)
\frac{\widetilde{g}_b(\xi)}
{4\Xi^2(x,\xi)}
e^{ix\cdot\xi}   d\xi
\end{equation*}
modulo lower order terms,
where $a_0(x,\xi) = \sigma(A)\big|_{\rho=0}$.

For the term, $\Gamma^{-1}_{int} \circ (
\rho  \tau v)$,
we use 
\begin{equation*}
\rho \tau v
=\rho \tau\circ \Theta^+ g
+\cdots,
\end{equation*}
where the $\cdots$ means lower order terms or
smoothing terms.  Hence,
modulo lower order terms, we have
\begin{align*}
\rho \tau v
=& \rho \tau\circ \Theta^+ g
\\
=& \rho \frac{i}{(2\pi)^{2n}}
 \sum_{jk} \int
\frac{\tau^{jk}_0(x)\xi_j\xi_k}
{\eta+i|\Xi(x,\xi)|} \widetilde{g}_b(\xi)
e^{i\rho\eta} e^{i x\xi} d\eta d\xi
\\
=&  \frac{1}{(2\pi)^{2n}}
 \sum_{jk}\int
\frac{\tau^{jk}_0(x)\xi_j\xi_k}
{(\eta+i|\Xi(x,\xi)|)^2} \widetilde{g}_b(\xi)
e^{i\rho\eta} e^{i x\xi} d\eta d\xi
,
\end{align*}
where $\tau_0^{jk}(x) := \tau_{jk}(x,0)$,
and thus
\begin{align*}
\Gamma^{-1}_{int} \left( \rho \tau v
\right)
=&  \frac{1}{(2\pi)^{2n}}
\sum_{jk}\int \frac{1}{\eta^2+\Xi^2(x,\xi)}
\frac{\tau^{jk}_0(x)\xi_j\xi_k}
{(\eta+i|\Xi(x,\xi)|)^2} \widetilde{g}_b(\xi)
e^{i\eta\rho} e^{i\xi x} d\eta d\xi
\\
=&  \frac{1}{(2\pi)^{2n}}
\sum_{jk}\int
\frac{1}{\eta-i|\Xi(x,\xi)|}
\frac{\tau^{jk}_0(x)\xi_j\xi_k}
 {\left(\eta+i|\Xi(x,\xi)|\right)^3}
 \widetilde{g}_b(\xi)
e^{i\eta\rho} e^{i\xi x} d\eta d\xi
,
\end{align*}
again, modulo lower order terms.
Integrating over $\eta$ and
setting $\rho=0$ yields
\begin{align*}
R\circ \Gamma^{-1}_{int}  \left( \rho \tau v
\right)
=&  -\frac{1}{(2\pi)^{2n-1}}
\frac{1}{8}  \sum_{jk}   \int 
 \frac{\tau^{jk}_0(x)\xi_j\xi_k}
{|\Xi(x,\xi)|^3}
\widetilde{g}_b(\xi) e^{i\xi x}  d\xi,
\end{align*}
modulo $\Psi^{-2}_bg_b$ and smoothing terms.

We can now read off the symbols
  homogeneous of degree -1
with respect to $|\xi|$ in
\eqref{homogE}:
\begin{align*}
0=& 
-\frac{1}{(2\pi)^{2n-1}}
 \frac{3i}{16}
\int  \frac{ \partial_{\xi}\Xi^2 \cdot \partial_x 
	\Xi^2}{ \Xi^4(x,\xi)}
\widetilde{g}_b(\xi)
e^{i\xi x}   d\xi
 +  \frac{1}{(2\pi)^{2n-1}}
\frac{1}{2}   \int 
 \frac{\widetilde{\Lambda_b^0 g}_b}
{|\Xi(x,\xi)|}e^{i\xi x}  d\xi\\
&- \frac{1}{(2\pi)^{2n-1}}
\frac{\sqrt{2}}{4}\int s_0(x)
\frac{\widetilde{g}_b(\xi)}
{|\Xi(x,\xi)|} e^{i\xi x} d\xi
-\frac{1}{(2\pi)^{2n-1}}
\frac{1}{4} \int a_{0}(x,\xi)
\frac{\widetilde{g}_b(\xi)}
{\Xi^2(x,\xi)}e^{i\xi x}   d\xi\\
&+ \frac{1}{(2\pi)^{2n-1}}
\frac{1}{8}    \sum_{jk} \int 
 \frac{\tau^{jk}_0(x)\xi_j\xi_k}
{|\Xi(x,\xi)|^3}
\widetilde{g}_b(\xi) e^{i\xi x}  d\xi.
\end{align*}

Solving for 
 $\sigma (\Lambda_b^0)(x,\xi)$
yields the
\begin{prop}
	\label{L0Prop}
	\begin{equation*}
	\sigma(\Lambda^0_b)
	= \frac{\sqrt{2}}{2} s_0(x)
	+ \frac{a_{0}(x,\xi)}{2|\Xi(x,\xi)|}
	- \frac{1}{4}
	\sum_{jk}
	\frac{\tau^{jk}_0(x)\xi_j\xi_k}
	{\Xi^2(x,\xi)} + 
	\frac{3i}{8}
	\frac{ \partial_{\xi}\Xi^2 \cdot \partial_x 
		\Xi^2}{ |\Xi(x,\xi)|^3}.
	\end{equation*}
\end{prop}
Finally, we can state the
\begin{thrm}
  \label{dno}
 Modulo pseudodifferential operators of order $-1$, the symbol for
 $N^-$ is given by
\begin{align*}
 \sigma(N^-) (x,\xi)
  =& |\Xi(x,\xi)|\\
  &+\frac{\sqrt{2}}{2} s_0(x)
  + \frac{a_{0}(x,\xi)}{2|\Xi(x,\xi)|}
  - \frac{1}{4}
  \sum_{jk}
  \frac{\tau^{jk}_0(x)\xi_j\xi_k}
  {\Xi^2(x,\xi)} + 
  \frac{3i}{8}
  \frac{ \partial_{\xi}\Xi^2 \cdot \partial_x 
  	\Xi^2}{ |\Xi(x,\xi)|^3}.
 \end{align*}
\end{thrm}
This is the same as
  Theorem 1.2 in \cite{CNS92}.

\section{The zero order term}
\label{SecZero}
 In this section we will look at the
zero order term of the DNO, and note
its possible vanishing under the hypothesis
 of a 
weakly pseudoconvex domain.  
 The vector field $(L_n - \Lb_n)/2i$
will play a special role in the following sections
 and the behavior of the boundary value operators
in its direction will be studied now.  We 
 use the terminology {\it
  transverse tangential} to refer to a 
  vector field which is tangential and
transverse to the complex tangent space
(also called the vector field of the 
 "missing direction" or the
  "bad direction" in the literature).

We start by recalling 
our notation
used in writing $N^-$.  Let 
 $N_1^-$ denote the operator 
which is given by the principal (first order)
 symbol of $N^-$, homogeneous of degree 1 in
$|\Xi(x,\xi)|$, where $\Xi(x,\xi)$ is given in
\eqref{xiLb}:
\begin{equation}
\label{xiLb2}
|\Xi(x,\xi)| = 
\sqrt{2\xi_{2n-1}^2+2\sum_{k<n}
	 \sigma(\Lb_{bk}) \sigma(L_{bk})}.
\end{equation}

 For the zero operator, we write the symbol
$a_0(x,\xi)$ in Theorem \ref{dno} according to
\begin{equation*}
 a_0(x,\xi) =
  \sum_{j=1}^{2n-1} \alpha_0^j(x) \xi_j.
\end{equation*}
From Theorem \ref{dno}, the zero order operator,
denoted by $N^-_0$, has symbol given by
\begin{align}
\sigma(N^-_0)=&
\frac{\sqrt{2}}{2}s_0(x)
+ \frac{1}{2}  \frac{\sum_{j=1}^{2n-1}
	\alpha^j_0(x)\xi_j}{|\Xi (x,\xi)|}
-\frac{1}{4}
\frac{\tau^{jk}_0(x)\xi_j\xi_k}{\Xi^2(x,\xi)}
\label{dno0}
+ \frac{3i}{8}
\frac{ \partial_{\xi}\Xi^2 \cdot \partial_x 
	\Xi^2}{ |\Xi(x,\xi)|^3},
\end{align}
in a neighborhood of a 
boundary point, which we assume
to be $0\in\partial\Omega$. 
Recall the functions $s_0$,
$\alpha_0^j$ and $\tau_0^{jk}$
as defined in Section \ref{dnosec}.

According to Proposition \ref{squareOp}, $s_0(x)$
is a diagonal matrix. 
(i.e. there are no normal derivatives of $u_I$ for $I\neq J$ which contribute to the term $f_J\omegab_J$ on the right hand side of
$\square u  = f_J \omegab_J$).
We have $s_{0,J}$, the diagonal $(J,J)$-entry of 
the matrix $s_0$, to be of the form
\begin{equation}
\label{sEqn}
s_{0,J}= -2i(-1)^{|J|}  \Im(c^J_{Jn})
+ d_n
\end{equation}
for $n\notin J$,
from Proposition \ref{squareOp} $i)$.

Now let $A_J$ denote the $J^{th}$ row of the matrix of first order operators, $A$.
We need (the vector product)
$A_J\cdot u$, so as to 
calculate the $J^{th}$ component
 of $Au$,
and in particular, we will need the
operators with the transversal tangential, $T$, component in the expression
$A_J\cdot u$ (in applying the results
 to the \dbar-Neumann condition
  in Section \ref{bndryEqn},  we are
interested in the behavior of the 
operators in a microlocal neighborhood
determined by $\psi^-$).

For the contribution of the sum
\begin{equation*}
- \sum_{l\notin J} L_l \Lb_l
- \sum_{l\in J} \Lb_lL_l
\end{equation*}
occurring in Proposition \ref{squareOp}
$i)$ to the $A$ operators,
we handle the case $l=n$ separately
 (again, assuming $n\notin J$):
\begin{align*}
-L_n\Lb_n =&
- \left(
\frac{1}{\sqrt{2}} \frac{\partial}{\partial\rho} + iT
\right)
\left(
\frac{1}{\sqrt{2}} \frac{\partial}{\partial\rho} - iT
\right) \\
=&-\frac{1}{2}
\frac{\partial^2}{\partial\rho^2}
- \frac{\partial^2}{\partial x_{2n-1}^2} +
\frac{i}{\sqrt{2}}\left[
\frac{\partial}{\partial\rho},T
\right]+ O(\rho)\Psi^2 \\
=&-\frac{1}{2}
\frac{\partial^2}{\partial\rho^2}
-\frac{\partial^2}{\partial x_{2n-1}^2}
+ \frac{i}{\sqrt{2}} T^1 + O(\rho)\Psi^2,
\end{align*}
where $T^1$ is defined to be
$\left[
\frac{\partial}{\partial\rho},T
\right]$ at $\rho=0$ (see also
\eqref{LnExpn} below).
We will be interested in the transverse tangential
component of the first order vector field at $\rho=0$
of $-2L_n\Lb_n$, that is,
in
\begin{equation}
\label{alphaT}
\frac{1}{|T^0|}
\left< \frac{2i}{\sqrt{2}} T^1 , \frac{T^0}{|T^0|}
\right>.
\end{equation}
 We use $<\cdot,\cdot>$ to denote the interior
product of two vector fields.  To ease notation
 we will also use the notation of the dot
product to denote the interior product in what
 follows.  Thus
\begin{equation*}
 \frac{1}{|T^0|}
  \frac{2i}{\sqrt{2}} T^1 \cdot
    \frac{T^0}{|T^0|}  :=
  \frac{1}{|T^0|}
  \left< \frac{2i}{\sqrt{2}} T^1 , \frac{T^0}{|T^0|}
  \right>
\end{equation*}
We could calculate this term explicitly, but we will not need to; it will eventually cancel out with another term in the DNO.

For $L_k \Lb_k$, $k\neq n$, we recall
 \eqref{Lj} and write
\begin{equation*}
\Lb_k\Big|_{\rho=0} = \frac{1}{2}\left(
\frac{\partial}{\partial x_{2k-1}}
+ i\frac{\partial}{\partial x_{2k}}
\right)
+\sum_{j=1}^{2n-1}
\overline{\ell}_{j}^k(x)
\frac{\partial}{\partial x_{j}}
,
\end{equation*}
where $\ell_j^k(x)=O(x)$.
We also use the representation
\begin{equation*}
\Lb_k
= \sqrt{2}
\sum \gamma^k_j
\frac{\partial}{\partial \zb_j},
\end{equation*}
where the $\gamma_j^k$ have the property that
\begin{equation*}
\sum_j \gamma_j^k 
\overline{\gamma}_j^l = \delta_{kl},
\quad 1\le k\le n-1
\end{equation*}
for the delta function 
$\delta_{kl}=1$ for $k=l$
and $\delta_{kl}=0$
for $k\neq l$,
and
\begin{equation*}
\sum_j \gamma_j^k 
\frac{\partial \rho}
{\partial \zb_j} = 0,
\quad 1\le k\le n-1.
\end{equation*}

At the boundary point $0\in\partial\Omega$, 
we can write
\begin{equation*}
\Lb_k\Big|_{p=0}=
\left. \frac{1}{2}\left(
\frac{\partial}{\partial x_{2k-1}}
+ i\frac{\partial}{\partial x_{2k}}
\right) \right|_{p=0}
= \sqrt{2}\sum \gamma^k_j(0)
\frac{\partial}{\partial \zb_j},
\end{equation*}
for $k\le n-1$.

As $\Lb_k\cdot T=0$ we have
\begin{align*}
\frac{1}{2} \overline{\ell}_{2n-1}^k(x)
=&- \sqrt{2}
\left(\sum \gamma^k_j(0)
\frac{\partial}{\partial \zb_j}\right) \cdot T + O(x^2)\\
=&- \frac{\sqrt{2}}{2i}
\left(\sum \gamma^k_j(0)
\frac{\partial}{\partial \zb_j}\right) \cdot \Lb_n + O(x^2)\\
=&  i 
\left(\sum \gamma^k_j(0)
\frac{\partial \rho}{\partial \zb_j}\right)  + O(x^2),
\end{align*}
where we use
$T\cdot \partial_{x_j}=O(x)$ for
$j\neq 2n-1$, 
$T\cdot T=\frac{1}{2}$ and
\begin{align*}
\Lb_n =&2\sqrt{2} \sum  
\frac{\partial \rho}{\partial z_j} \frac{\partial}{\partial \zb_j}.
\end{align*}
Hence,
\begin{equation*}
\overline{\ell}_{2n-1}^k(x)
= 2 i \sum \gamma^k_j(0)
\frac{\partial \rho}{\partial \zb_j}
+O(x^2).
\end{equation*}

With 
\begin{equation*}
L_k =\sqrt{ 2} \sum \gammab_l^k \frac{\partial}{\partial z_l},
\end{equation*}
we can write
the coefficient of the transverse tangential
vector field, $T$, in the expression
from $-L_k \Lb_k$ as
\begin{align}
\nonumber
- L_k(\overline{\ell}_{2n-1}^k(x))
=&
- 2i  \sum \gamma^k_j(0)
L_k\left(\frac{\partial \rho}{\partial \zb_j}\right) + O(x)\\
\nonumber
=& - 2i\sqrt{2} \sum \gamma^k_j(0)
\gammab^k_l
\frac{\partial^2 \rho}{\partial z_l \zb_j}
+O(x)\\
\label{LEll}
=&- i\sqrt{2} |L_k|_{\lrl}^2 + O(x),
\end{align}
where $|\cdot|_{\lrl}$ refers to the
length with respect to the Levi metric, which 
 is
define by
\begin{equation*}
ds^2 = 
\sum \frac{\partial^2 \rho}{\partial z_l \zb_j} dz_l d\zb_j.
\end{equation*}

That is,
\begin{equation*}
-2L_k \Lb_k
= - \frac{1}{2}\left(
\frac{\partial^2}{\partial x_{2k-1}^2}
+ \frac{\partial^2}{\partial x_{2k}^2}
\right)
- i2\sqrt{2} |L_k|_{\lrl}^2  \frac{\partial}{\partial x_{2n-1}}
+ \cdots
\end{equation*}
where the $\cdots$ refer to 
second order terms with coefficients
in $O(x)$ or first order terms which
upon contraction with $T$ result in 
$O(x)$ functions.
And similarly,
\begin{equation*}
-2\Lb_k L_k
= - \frac{1}{2}\left(
\frac{\partial^2}{\partial x_{2k-1}^2}
+ \frac{\partial^2}{\partial x_{2k}^2}
\right)
+ i2\sqrt{2} |L_k|_{\lrl}^2  \frac{\partial}{\partial x_{2n-1}}
+ \cdots.
\end{equation*}

Thus the transverse tangential component 
to be included in the operator
$A$
of the first order vector fields
from 
$ - 2\sum_{k\notin J} L_k \Lb_k
- 2\sum_{k\in J} \Lb_kL_k$ 
written in our local coordinates
is given by
\begin{equation}
\label{tangT}
i2\sqrt{2} \left(
\sum_{k\in J} |L_k|_{\lrl}^2
-  \sum_{k\notin J} |L_k|_{\lrl}^2
\right).
\end{equation}

From the first order operators in Proposition
\ref{squareOp}, we see there are also the $T$ components
to be included in the operator $A$ given by
\begin{equation}
\label{TFromProp}
-(-1)^{|J|} 4i \Re(c^J_{Jn}) -
2id_n .
\end{equation}

We now move to the operator $\tau$.
From Proposition \ref{squareOp},
$\tau$ is a diagonal operator.
Let us calculate the asymptotic behavior of the 
entries of the symbol of $\tau$ 
for large $|\xi_{2n-1}|$.
Recall that in the $\tau$ operator,
we collected all the second order
tangential derivatives with
coefficients which are $O(\rho)$.  

We expand    
\begin{equation}
\label{LnExpn}
T=T^0 + \rho T^1 + \rho^2 T^2
+\cdots
\end{equation} 
and
\begin{equation*}
L_j=L_j^0 + \rho L_j^1 + \rho^2 L_j^2
+\cdots
\end{equation*} 
for $1\le j\le n-1$.

The second order $O(\rho)$ 
operators arise from
\begin{align*}
- 2L_n \Lb_n =& -2\left(
\frac{1}{\sqrt{2}} \frac{\partial}{\partial\rho}
+ i(T^0 + \rho T^1 + \cdots) \right)
\left(
\frac{1}{\sqrt{2}} \frac{\partial}{\partial\rho}
- i(T^0 + \rho T^1 + \cdots) \right) \\
=& \cdots - 2\rho T^1 T^0 - 2\rho T^0 T^1 + \cdots
\end{align*}
as well as
\begin{align*}
-2L_j \Lb_j =& -2\left(
L_j^0 + \rho L_j^1+\cdots \right)
\left(
\Lb_j^0 + \rho \Lb_j^1+\cdots \right) \\
=& \cdots - 2\rho L_j^1\Lb_j^0
- 2\rho L_j^0 \Lb_j^1 + \cdots.
\end{align*}

We specifically want, $\tau^{2n-1,2n-1}$, the coefficient
of $\frac{\partial^2}{\partial x_{2n-1}^2}$ in (the diagonal components of) $\tau$.  
Thus, for instance, from the $L_n\Lb_n$,
$\tau^{2n-1,2n-1}$ contains
the coefficient of
\begin{equation*}
-2\left[\left(T^1 \cdot
\frac{T^0}{|T^0|}\right)
\frac{T^0}{|T^0|}\right] T^0 - 
2T^0 \left[\left(
T^1 \cdot \frac{T^0}{|T^0|}
\right) \frac{T^0}{|T^0|}\right],
\end{equation*}   i.e., 
\begin{equation}
\label{t1t0}
\tau^{2n-1,2n-1}
= - 4\sqrt{2} \left< T^1, \frac{T^0}{|T^0|} \right>,
\end{equation}
since there are no contributions from
$-2L_j\Lb_j$, in the form
of $-2L_j^1\Lb_j^0$ and
$-2L_j^0\Lb_j^1$, due to the property that
$L_j\cdot T=0$. 
As we mentioned earlier, we will have no need
  to calculate explicitly the interior product.

Furthermore, we can handle the last term
in \eqref{dno0} by noting that
\begin{align}
\nonumber
\frac{\sum \partial_{x_j}\Xi^2(x,\xi)
	\partial_{\xi_j}\Xi^2(x,\xi)}
{|\Xi(x,\xi)|^3}
= & \frac{\left( 
	O(|\xi_L||\xi_{2n-1}|)
	+ O(\xi_L^2) +
	O(x)O(\xi_{2n-1}^2)\right)
	\cdot O(\xi)}
{|\Xi(x,\xi)|^3}\\
\label{derXiderXi}
=& O\left(\frac{|\xi_L|}{|\Xi(x,\xi)|}
\right)
+ O(x),
\end{align}
for large $\xi_{2n-1}$.

Lastly, to handle the non-diagonal 
 terms in $\sigma(N^-_0)$ we consider
the transverse tangential components of
 the terms in Proposition \ref{squareOp}
$ii)$.  Noting that 
 $[L_j,\Lb_k] \cdot T$ give the entries for
the Levi matrix, and assuming without loss of
 generality that the Levi matrix is diagonal
 (at the given point $0\in\partial\Omega$),
the contributions of such components in 
 the transverse tangential direction are 
  $O(x)$.
The non-diagonal terms in 
 $\sigma(N^-_0)$ are thus in 
$O\left(
\frac{|\xi_L|}{|\Xi(x,\xi)}
\right)+O(x)$. 

 In the expression for the
zero order term of $N^-$ we write
 $(b_J)_J$ to mean the diagonal matrix
whose $(J,J)^{th}$ entry is given by 
 $b_J$.  All terms in the expression 
for $\sigma(N^-_0)$, with the exception
 of error terms, will be diagonal matrices.
Using \eqref{sEqn}, 
\eqref{alphaT}, \eqref{tangT},
\eqref{TFromProp},
and \eqref{t1t0} in the expression for
$\sigma(N^-_0)$ in \eqref{dno0},
 and restricting to the boundary, we have
\begin{prop}
	\label{propertyDNOSquare}
	Let $\sigma(N_0^-)$ be the zero order 
	symbol in the expansion of the DNO
	associated with the $2\square$ operator.
	Then in a microlocal neighborhood of the 
	boundary point $0\in\partial \Omega$
	for large $|\xi_{2n-1}|$ we have
	\begin{align*}
	\sigma(N^-_0)=&
	\frac{\sqrt{2}}{2}
	\left(
	-2i(-1)^{|J|}  \Im(c^J_{Jn})
	+ d_n
	\right)_J\\
	&+\left(
	(-1)^{|J|} 2 \Re(c^J_{Jn}) +
	d_n 
	-  \left< T^1 , \frac{T^0}{|T^0|}
	\right> \right)_J
	\frac{\xi_{2n-1}}{|\Xi(x,\xi)|}
	\\
	&-\sqrt{2}
	\left(
	\sum_{k\in J} |L_{bk}|_{\lrl}^2
	-  \sum_{k\notin J} |L_{bk}|_{\lrl}^2
	\right)_J
	\frac{\xi_{2n-1}}{|\Xi (x,\xi)|}
	\\
	&-
	\sqrt{2}
\left(
	\left< T^1, \frac{T^0}{|T^0|} \right>
\right)_J
	\frac{\xi_{2n-1}^2}{\Xi^2(x,\xi)}
	+O\left(
	\frac{|\xi_L|}{|\Xi(x,\xi)}
	\right) 
	+ O(x).
	\end{align*}
\end{prop}

\section{Boundary equation}
\label{bndryEqn}

The \dbar-Neumann problem for a $(0,q)$-form
 $f\in L^2_{(0,q)}(\Omega)$ is to find a solution
$u\in L^2_{(0,q)}(\Omega)$ to
 $\square u = f$.  As the $\square$ operator
consists of $\mdbar^{\ast}$ operators, boundary
 conditions arise on $u$ so as to
  fulfill conditions
regarding its inclusion in 
  the domain of $\mdbar^{\ast}$.  The 
 \dbar-Neumann is the boundary value problem
\begin{equation*}
 \square u = f \qquad \mbox{in } \Omega
\end{equation*}
with boundary conditions
\begin{align*}
& u \rfloor \mdbar \rho = 0\\
& \mdbar u \rfloor \mdbar \rho = 0
\end{align*}
 on $\partial\Omega$.

The first boundary condition
 $u\rfloor\mdbar \rho =0$ is just
$u_{J} = 0$ on $\partial\Omega$ for any 
 $J$ such that $n\in J$.
For the second condition involving 
 $\mdbar u$, we 
note
\begin{equation*}
 \mdbar u = \sum_{J\not\owns n}
\left(
  (-1)^{|J|} \Lb_n u_J +
c^J_{Jn} u_J+
   \varepsilon^{kJ_{\hat{k}n}}_{Jn}
\Lb_k u_{J_{\hat{k}}n}
 \right) \omegab_{Jn} + \cdots,
\end{equation*} 
where $\cdots$ refers to terms with no
 $\omegab_n$ component.  
 
Assuming the boundary condition
 $u\rfloor\mdbar\rho=0$,
we have
$\mdbar u \rfloor \mdbar \rho = 0$
 is equivalent to 
\begin{equation}
 \label{bcLn}
 \Lb_n u_J + (-1)^{|J|}c^J_{Jn} u_J = 0
\end{equation}
 on $\partial\Omega$
for $J$ such that $n\notin J$.
 
We  
 write the solution $u$
in terms of a Green's solution and Poisson
 solution:
\begin{equation*}
 u = G(2f) + P(u_b)
\end{equation*} 
where the operators $G$ and $P$ satisfy
 \begin{equation}
\label{GError}
\begin{aligned}
&2\square\circ G = I\\
&R\circ G = 0
\end{aligned}
\end{equation}
and
 \begin{equation}
\label{PError}
\begin{aligned} 
&2\square\circ P = 0\\
&R\circ P = I,
\end{aligned}
\end{equation}
respectively.

The $J^{th}$ component will be written
 $u_J = G_J(2f) + P_J(u_b)$.  
$\Lb_n u_J$ can now be written
 on $\partial\Omega$ as
\begin{equation*}
 R\circ \Lb_n u_J = 
  \frac{1}{\sqrt{2}}
 R\circ \partial_{\rho} \circ
   G_J(2f) + 
\left(   
   \frac{1}{\sqrt{2}} N^-
   -iT^0
\right) u_{b,J},
\end{equation*}
where $u_{b,J}$ is the $J^{th}$ component of
 $u_b$.

From \cite{Eh18_halfPlanes}
 (Theorem 3.3), we use the 
 the property that
\begin{equation*}
 R\circ \partial_{\rho} \circ
 G_J \equiv R\circ \Psi^{-1}
\end{equation*}
 modulo smoothing terms.  
The boundary condition \eqref{bcLn}
 for $J\not\owns n$ can therefore be written
as
\begin{equation*}
\left( \frac{1}{\sqrt{2}} N^-_1
-i T^0\right) u_{b,J} + (-1)^{|J|} c^J_{Jn} u_{b,J}
+ \frac{1}{\sqrt{2}}
\left(N^-_{0}u_{b} \right)_J
= R\circ \Psi^{-1} f,
\end{equation*}
modulo lower order terms in
 $u_b$.  
As mention above in Section \ref{SecZero},
we will concentrate on the microlocal
 region determined by the
symbol, $\psi^{-}$, that is, 
 the region in which (in local coordinates)
$\xi_{2n-1}$ is large and negative.  
 The reason is that in the other regions,
estimates for $u_b$ can be obtained by
 inverting the operator,
$1/\sqrt{2} N_1^- -iT^0$.

We will need the behavior of the 
operators in the microlocal neighborhood
of a boundary point, $0\in\partial\Omega$
and with support in the support of
the symbol $\psi^{-}$.  
To this end, we first consider the limit of $N^-_0$
as $\xi_{2n-1}\rightarrow -\infty$.

Let us write
\begin{equation*}
 \left(N^-_{0}u_{b} \right)_J
  = N^-_{0,J}u_{b,J} + 
   N^-_{0,JX}u_{b},
\end{equation*}
where $N^-_{0,J}$ is the $(J,J)$ entry in 
 the matrix $N_0^-$ and
$ N^-_{0,J_X}$ is the matrix consisting of
 the $J^{th}$ row of $N_0^-$, with the
$(J,J)$ entry replaced with 0, and zeros 
 elsewhere. 

From Proposition 
\ref{propertyDNOSquare}, 
\begin{equation*}
\sigma\left(N^-_{0,J_X}\right)
 = O\left(
 \frac{|\xi_L|}{|\Xi(x,\xi)}
 \right) 
 + O(x),
\end{equation*} 
and as
$\xi_{2n-1}\rightarrow -\infty$,
we see
\begin{align}
\nonumber
\sigma(N^-_{0,J} ) \rightarrow&
\frac{\sqrt{2}}{2}
\left(  -2i(-1)^{|J|}  \Im(c^J_{Jn})
+ d_n
\right)\\
\nonumber
&   -\frac{1}{\sqrt{2}}
\left(
(-1)^{|J|} 2 \Re(c^J_{Jn}) +
d_n
- \left<  T^1 , \frac{T^0}{|T^0|}
\right>
\right)\\
\nonumber
& +  
\sum_{k\in J} |L_{bk}|_{\lrl}^2
-  \sum_{k\notin J} |L_{bk}|_{\lrl}^2
- \frac{1}{\sqrt{2}}
\left<  T^1 , \frac{T^0}{|T^0|}
\right> + O(x)+O\left(
\frac{|\xi_L|}{|\Xi(x,\xi)}
\right)   \\
\nonumber
=& -(-1)^{|J|} \sqrt{2} c^J_{Jn}
+  
\sum_{k\in J} |L_{bk}|_{\lrl}^2
-  \sum_{k\notin J} |L_{bk}|_{\lrl}^2\\
\label{zeroOrderOrig}
&
+O(x) +O\left(
\frac{|\xi_L|}{|\Xi(x,\xi)}
\right)    .
\end{align}
We could also at this point proceed to
calculate each of the
$c_{Jn}^J$, but as we will see, these will also cancel in what follows.
We will denote the zero order operator
$ (-1)^{|J|} c^J_{Jn}
+ \frac{1}{\sqrt{2}}
N^-_{0} $ (with $ c^J_{Jn}$ referring to
the operator with a single diagonal entry) by $\Upsilon^0_J$.  From above we have
\begin{align*}
\nonumber
\sigma \left(\Upsilon^0_J \right)
\rightarrow&  (-1)^{|J|} c^J_{Jn} + \frac{1}{\sqrt{2}}
\left(-(-1)^{|J|} \sqrt{2} c^J_{Jn}
+  
\sum_{k\in J} |L_{bk}|_{\lrl}^2
-  \sum_{k\notin J} |L_{bk}|_{\lrl}^2
\right)
\\
&+O(x) +O\left(
\frac{|\xi_L|}{|\Xi(x,\xi)}
\right)   \\
=&\frac{1}{\sqrt{2}}
\left(
\sum_{k\in J} |L_{bk}|_{\lrl}^2
-  \sum_{k\notin J} |L_{bk}|_{\lrl}^2
\right) +O(x)+O\left(
\frac{|\xi_L|}{|\Xi(x,\xi)}
\right)   ,
\end{align*}
as $\xi_{n-1}\rightarrow
- \infty$, recalling that
$N_{0,J_X}^- =O(x)+O\left(
\frac{|\xi_L|}{|\Xi(x,\xi)}
\right)$.

We collect our results in the 
following Proposition
\begin{prop}
	 \label{propNMinusT}
 The boundary equation for the
  \dbar-Neumann problem has the form
\begin{equation}
\label{nMinusTUpOrig}
\left( \frac{1}{\sqrt{2}} N^-_1
-i T^0\right) u_{b,J} +
\Upsilon^0_J u_{b}
=  R\circ \Psi^{-1}f,
\end{equation}
where
\begin{equation}
\label{upDef}
 \Upsilon^0_J u_{b}
  = \Upsilon^0_{J,J} u_{b,J}
   + \sum_{K\neq J} \Upsilon^0_{J,K} u_{b,K},
\end{equation}
and
 $\Upsilon^0_{J,J}$ is a psedodifferential operator
  of order 0,
 whose symbol has the property
\begin{equation}
 \label{upProp1}
 \sigma(\Upsilon^0_{J,J})=
  \frac{1}{\sqrt{2}}
 \left(
 \sum_{k\in J} |L_{bk}|_{\lrl}^2
 -  \sum_{k\notin J} |L_{bk}|_{\lrl}^2
 \right) +O(x)+O\left(
 \frac{|\xi_L|}{|\Xi(x,\xi)}
 \right) 
\end{equation}
and
 $\Upsilon^0_{J,K}$
  is a psedodifferential operator
of order 0,
whose symbol has the property
\begin{equation}
 \label{upProp2}
\sigma(\Upsilon^0_{J,K})=
O(x)+O\left(
\frac{|\xi_L|}{|\Xi(x,\xi)}
\right) 
\end{equation}
for $K\neq J$.
\end{prop}

At this point, we take a moment to
review how previous work on inverting
the Kohn Laplacian, $\square_b$, defined
on the boundary, could be useful in
 solving \eqref{nMinusTUpOrig}.   
An inverse to $\square_b$ in the case of
strictly pseudoconvexity was studied in 
\cite{FoSt74}, and we first
relate our boundary equation
 \eqref{nMinusTUpOrig} to that of
\cite{FoSt74}.
We simplify our equation, throwing away
the $O(x)$ and $O\left(
\frac{|\xi_L|}{|\Xi(x,\xi)}
\right)$ terms (for the purpose of illustration 
only) and consider
\begin{equation*}
\left( \frac{1}{\sqrt{2}} N^-_1
-i T^0\right) u_{b,J} +
\mathscr{Y}^0_J u_{b,J}
=  R\circ \Psi^{-1}f,
\end{equation*} 
with
\begin{equation*} 
 \sigma(\mathscr{Y}^0_J)
  = \frac{1}{\sqrt{2}}
  \left(
  \sum_{k\in J} |L_{bk}|_{\lrl}^2
  -  \sum_{k\notin J} |L_{bk}|_{\lrl}^2
  \right).
\end{equation*}

We now apply the operator
$\frac{1}{\sqrt{2}} N^-_1
+i T^0$ to both sides:
\begin{align}
\nonumber
\left( \frac{1}{2} (N^-_1)^2
+ (T^0)^2\right) u_{b,J} +
&\frac{i}{\sqrt{2}} [T^0 ,N^-_1]u_{b,J}\\
&+ \left(\frac{1}{\sqrt{2}} N^-_1
+i T^0
\right)\circ \mathscr{Y}^0_J u_{b,J}
\label{NPlusT}
= R\circ
\Psi^{0}f.
\end{align}

We first note some properties of the operators involved.
Consider the first order operator
$[T^0 ,N^-_1]$.  
By expanding the symbol
for $N_1^-$ for large
$|\xi_{2n-1}|$,
we see (for large $|\xi_{2n-1}|$)
\begin{align*}
\sigma([T^0 ,N^-_1])
=& \partial_{x_{2n-1}}
\sigma\left(N^-_1 \right)\\
=& O(|\xi_L|)+O(x)O(|\xi|)
\end{align*}
modulo $\lrs^{-\infty}(\partial\Omega)$.
We also write the operator
$ \frac{1}{2} (N^-_1)^2
+ (T^0)^2$ in terms of the vector fields
$L_j$ and $\Lb_j$.  The symbol of $(N^-_1)^2$
is given by
\begin{align}
\nonumber
\sigma\left((N^-_1)^2\right)
= & \sigma(N^-_1)\sigma(N^-_1)
-i\partial_{\xi}\sigma(N^-_1)
\cdot \partial_{x}\sigma(N^-_1)
+\cdots\\
\label{SymN2}
=& \Xi^2(x,\xi) + 
O(|\xi_L|)+O(x)
\end{align}
modulo $\lrs^{-1}(\partial\Omega)$.
For the term $ \Xi^2(x,\xi)$ we have
 from \eqref{xiLb}
\begin{equation*}
\Xi^2(x,\xi)=
2\xi_{2n-1}^2
+ 2\sum_j \sigma(\Lb_{bj}) \sigma(L_{bj})
\end{equation*}
and for $\sigma(\Lb_{bj}) \sigma(L_{bj})$
  we have the relations
\begin{align*}
\sigma(\Lb_{bj}L_{bj}) =&
\sigma(\Lb_{bj}) \sigma(L_{bj})
-i\sum_j \partial_{\xi} \sigma(\Lb_{bj}) \cdot
\partial_{x}\sigma(L_{bj})\\
=& \sigma(\Lb_{bj}) \sigma(L_{bj})
+\Lb_{bj}(\ell_{2n-1}^j) (i\xi_{2n-1})
+ O(x)O(|\xi|)  \\
=&   \sigma(\Lb_{bj}) \sigma(L_{bj})
+ \sqrt{2}\xi_{2n-1}  |L_{bj}|^2_{\lrl}
+  O(x)O(|\xi|) ,
\end{align*}
modulo $\xi_L$ terms and
symbols of class $\lrs^0$,
where we use \eqref{LEll} in the last line.

Similarly, we have for
$ \sigma(\Lb_{bj})\sigma(L_{bj})$
\begin{equation*}
\sigma(L_{bj}\Lb_{bj}) =
     \sigma(L_{bj})\sigma(\Lb_{bj})
- \sqrt{2}\xi_{2n-1}  |L_{bj}|^2_{\lrl}
+  O(x)O(|\xi|) +\cdots.
\end{equation*}

Then the expression for
 $\Xi^2(x,\xi)$ gives
\begin{multline*}
\Xi^2(x,\xi)= 2\xi_{n-1}^2
+ 2\sum_{k\notin J} \sigma(\Lb_{bk}L_{bk})
+ 2\sum_{k\in J} \sigma(L_{bk}\Lb_{bk}) \\
- 2\sqrt{2}\xi_{n-1} \sum_{k\notin J}
|L_{bk}|^2_{\lrl}
+ 2\sqrt{2}\xi_{n-1} \sum_{k\in J} |L_{bk}|^2_{\lrl}
+  O(x)O(|\xi|)
+O(|\xi_L|),
\end{multline*}
which, combined with the expression
in \eqref{SymN2} above, yields
(for the $(J,J)$-entry)
\begin{multline*}
\sigma\left(
\frac{1}{2}(N^-_1)^2
+ (T^0)^2\right)
= \sum_{k\notin J} \sigma(\Lb_{bk}L_{bk})
+ \sum_{k\in J} \sigma(L_{bk}\Lb_{bk})\\
- \sqrt{2}\xi_{n-1} \sum_{k\notin J}
|L_{bk}|^2_{\lrl}
+ \sqrt{2}\xi_{n-1} \sum_{k\in J} 
 |L_{bk}|^2_{\lrl}
+  O(x)O(|\xi|)
+O(|\xi_L|).
\end{multline*}

Furthermore,
\begin{align*}
\sigma\left(
\left(\frac{1}{\sqrt{2}} N^-_1
+i T^0
\right)\circ \mathscr{Y}^0_J \right)
=& \sqrt{2}|\xi_{n-1}|
\left(
\sum_{k\in J} |L_k|_{\lrl}^2
-  \sum_{k\notin J} |L_k|_{\lrl}^2
\right)\\
&+ O(|\xi_L|)
+O(x)O(|\xi|)
\end{align*}
for $|\xi_L|\ll |\xi_{2n-1}|$,
 and $\xi_{2n-1} <0$,
modulo lower order symbols.

\eqref{NPlusT} is thus reduced to studying
\begin{equation*}
\sum_{k\notin J}\Lb_{bk}L_{bk}
+ \sum_{k\in J} L_{bk}\Lb_{bk}
\end{equation*}
modulo first order operators with
symbols which can be made arbitrarily small in a 
microlocal neighborhood of the boundary point
$0\in\partial\Omega$ for
$|\xi_L|\ll \xi_{2n-1}$.

In the highest order, this is just the
Kohn Laplacian, $\square_b$ which, under 
the hypothesis of strict pseudoconvexity,
can be inverted 
by analyzing the operator on the
Heisenberg group, as in \cite{FoSt74},
or in the case of finite type by considering 
relations of 
commutators of the vector fields, $L_k$
and their conjugates, as in \cite{RoSt}. 
The problem in the 
case of weak pseudoconvexity is that
the
means to control
derivatives in the 
direction of $T$, namely through commutators of
the vector fields, $L_j$, 
with vector fields, $\Lb_k$, 
is no longer available.

One of the immediate difficulties in using
 the method of applying the boundary 
operator $\frac{1}{\sqrt{2}}N_1^- + iT^0$ 
as above leading to \eqref{NPlusT} is that
 the resulting highest order symbol,
\begin{equation*}
 \sigma\left(
  \frac{1}{2}(N_1^-)^2 + (T^0)^2
 \right)
\end{equation*}
is not elliptic.  It is missing estimates from
 below by the $\xi_{2n-1}$ transform variable.
In other words, an estimate of the form
\begin{equation*}
 \sigma\left(
 \frac{1}{2}(N_1^-)^2 + (T^0)^2
 \right)
   \gtrsim 1 + |\xi_{2n-1}|^2
\end{equation*}
for $|\xi_L| \gg 1$ does not hold.
 It still may be possible to obtain information
of the solution to \eqref{nMinusTUpOrig}
 if it were possible to obtain a lower
order estimate, an estimate of the
 first order terms of \eqref{NPlusT} of the
  form
 \begin{equation*}
 \label{missingEstimate}
  \sigma\left[\left(
	  \frac{1}{\sqrt{2}} N_1^- + iT^0
  \right)\circ \Upsilon^0_J \right]
   \gtrsim 1+ |\xi_{2n-1}|
 \end{equation*}
and use the missing first order estimate as a 
 (weaker) substitute for an elliptic
second order estimate.
 This idea is used in 
\cite{Eh18_pwSmth} to obtain
 (weighted) estimates of the boundary 
solution.
 
The aim of the next sections is to show how
 persistent the absence of ellipticity
in the boundary equation is.

\section{Variations of the $\square$
 operator}
\label{varSquare} 
 In this section we consider
operators obtained from the $\square$
 operator by adding additional terms.
 In particular, we let $\phi$ be 
a function supported near the boundary and
with 
 $\square_{\phi}
 = \mdbar \mdbar^{\ast}
 + \mdbar^{\ast}\mdbar\circ (1+\phi)$, 
 we consider the boundary value problem:
\begin{equation*}
\square_{\phi} u
= f
\end{equation*}
with the boundary conditions,
 \begin{equation}
\label{modBC}
\begin{aligned}
  u \rfloor \mdbar \rho =& 0,\\
 \mdbar  \big((1+\phi)u
\big) \rfloor \mdbar \rho =& 0,
\end{aligned}
\end{equation}
holding on $\partial\Omega$.
 The first condition ensures
$u\in \mbox{dom}(\mdbar^{\ast})$
 and the second that 
$\mdbar\big((1+\phi)u \big)
 \in \mbox{dom}(\mdbar^{\ast})$.
 
We first look at the case
 $\phi$ only depends on $\rho$:
$\phi=\phi(\rho)$, and
 $\phi(0)=0$, and we use the notation
  from the previous sections.
 In this case the condition
$\mdbar  \big((1+\phi)u
\big) \rfloor \mdbar \rho = 0$ can be written
\begin{equation*}
 \sum_k (1+\phi) \Lb_k u_{J_{\hat{k}}n }
  + (1+\phi) (-1)^{|J|}\Lb_n u_J
  +(-1)^{|J|}(\Lb_n \phi) u_J 
  + (1+\phi) c_{Jn}^J u_J
  =0 .
\end{equation*}
Combined with the first boundary condition,
 $u\rfloor \mdbar \rho =0$, and recalling
  $\phi(0)=0$, this yields
\begin{equation}
 \label{bndryRaw}
 \Lb_n u_J
 +(\Lb_n \phi) u_J 
 + (-1)^{|J|} c_{Jn}^J u_J=0 .
\end{equation}

At first sight, a hold on regularity appears
 possible, in light of the discussion at the
end of Section \ref{bndryEqn}, as the term
 $\Lb_n \phi$ allows for a strictly positive
(diagonal)
addition to the $\Upsilon^0_J$ operator.  
 We repeat the steps of the previous sections
to obtain an expression of \eqref{bndryRaw}
 in terms of the 
 complex tangential vector fields, $L_j$; as 
before, the main calculation concerns the
 DNO.    

To recall,
we write $u_J$ as a sum of solutions to 
 Dirichlet problems, the solutions written in 
terms of Green's operator and a Poisson 
 operator
(for the analogues to the systems,
 \eqref{GError} and \eqref{PError},
with $\square$ replaced by $\square_{\phi}$):
\begin{equation*}
 u = G^{\phi}(2f) + P^{\phi}
  (u_{b}).
\end{equation*}
 Also, we have
\begin{align*}
  \Lb_n P^{\phi}(u_b)\Big|_{\rho=0} = &
   \left( \frac{1}{\sqrt{2}} \partial_{\rho}
    -iT \right) P^{\phi}(u_b)\Big|_{\rho=0} \\
=& \frac{1}{\sqrt{2}}  N^{\phi,-} u_b 
   -iT^0 u_b,
\end{align*}
and
\begin{align*}
 \Lb_n G^{\phi}(2f) \Big|_{\rho=0}
  =& \frac{1}{\sqrt{2}} \partial_{\rho}
   G^{\phi}(2f) \Big|_{\rho=0}\\
  =& R\circ \Psi^{-1} f.
\end{align*}
We can now rewrite \eqref{bndryRaw} as
\begin{equation}
 \label{bndryExpnd}
\left(\frac{1}{\sqrt{2}}  N^{\phi,-}
-iT^0 \right) u_{bJ}
 + \left( \frac{1}{\sqrt{2}} \phi'(0)
  +(-1)^{|J|}c_{Jn}^J \right) u_{bJ}
  = R\circ \Psi^{-1}f.
\end{equation}

As in the case with $\phi\equiv 0$,
 the operator 
 $\frac{1}{\sqrt{2}}  N^{\phi,-}
 -iT^0$ is of first order, but 
  it is not elliptic
since its principal symbol,
\begin{equation*}
 \frac{1}{\sqrt{2}}  |\Xi(x,\xi)|
 + \xi_{2n-1}
\end{equation*}
 tends to 0 
as $\xi_{2n-1} \rightarrow -\infty$.
However, a non-vanishing 
zero order term in the 
symbol expansion of $N^{\phi,-}$ would,
 after composition with 
$\frac{1}{\sqrt{2}}  N^{\phi,-}
+iT^0 $
 lead to a first order term whose
symbol is non-vanishing in the
 support of $\psi^-$.  
 We thus examine the term
$\sigma_0(N^{\phi,-})$.  

For 
 $\phi=0$, we have Theorem 
  \ref{dno}.  In the case 
 $\phi=\phi(\rho) \neq 0$ we examine the changes
induced on the DNO.  
 To highest order, the operators
$\square_{\phi}$ and $\square$ are
 identical, so we determine the 
operators,
 $S_{\phi}$,
  $A_{\phi}$,
and $\tau_{\phi}$,
 (and their corresponding
 symbols) with
which $\square_{\phi}v$
  can be written 
 as in \eqref{DirProbCorrOps}
 in the interior of
   $\Omega$ as
\begin{equation*}
\Gamma v
+ \sqrt{2}S_{\phi} \left(\frac{\partial v}{\partial\rho}\right)
+ A_{\phi} v+\rho\tau_{\phi}(v)
=0.
\end{equation*}

We write the operator $\square_{\phi}$
 in local coordinates.  We have
\begin{equation*}
 \square_{\phi}
  = \square + \mdbar^{\ast}\mdbar \phi.
\end{equation*}
In
Proposition \ref{squareOp},
we examined the forms which, upon
action through the $\square$ operator
would result in terms with a certain
$\omegab_J$ component. 
We follow this same approach 
here for the operator $\square_{\phi}$
in order to obtain an expression 
for the DNO corresponding to the
$\square_{\phi}$ operator.  

We examine the
term $\mdbar^{\ast}\mdbar( \phi u)$.
We have
\begin{equation}
\label{dbarRhoT}
\mdbar ( \phi u_J\omegab_J )
= 
(-1)^{|J|} (\Lb_n\phi)u_J \omegab_J \wedge
\omegab_n
+\phi(\rho)
\sum_J (-1)^{|J|} ( \Lb_n u_J) \omegab_J \wedge
\omegab_n + \cdots,
\end{equation}
where the $\cdots$ refer to forms with no 
 $\omegab_n$ component, or, in the 
case $n\in J$ involve only the 
 $\Lb_j$ 
 operators for $j=1,\ldots,n-1$.

And from
\begin{equation*}
\mdbar^{\ast} v \omegab_J \wedge
\omegab_n=
\left( (-1)^{|J|} \left(-L_n + d_n  \right) +\overline{c}_{J \cup\{n\}}^J \right) v \omegab_{J}+\cdots,
\end{equation*}
 where here the $\cdots$ denote terms which
are orthogonal to $\omegab_J$,
we get
\begin{align*}
\mdbar^{\ast} \mdbar ( \phi& u_J\omegab_J )\\
=&\bigg( -\phi(\rho) L_n\Lb_n u_J -
 (\Lb_n\phi) L_n u_J
-
 (L_n\phi)\Lb_n u_J\\
&+ \phi(\rho)\left(
d_n + (-1)^{|J|} \overline{c}_{J \cup\{n\}}^J
\right)\Lb_n u_J\bigg) \omegab_J+\cdots\\
=&  \bigg(-\phi(\rho) L_n\Lb_n u_J
-
 \phi'(\rho) \partial_{\rho} u_J
+ \phi(\rho)\left(
d_n + (-1)^{|J|} \overline{c}_{J \cup\{n\}}^J
\right)\Lb_n u_J\bigg)\omegab_J + \cdots.
\end{align*}
Again, for the error terms we include all
 0 order terms, terms orthogonal
to $\omegab_J$, and terms involving only
$L_j$ and/or $\Lb_j$ for $j=1,\ldots, n-1$
 in the $\cdots$.

From \eqref{dstard},
we have
\begin{equation*}
\label{extraTerm}
\mdbar^{\ast} \mdbar ( \phi(\rho)
u_{kl}
\omegab_{J_{\hat{k}}\cup\{l\}}
)
=  -\phi(\rho) \varepsilon^{lJ}_{J\cup\{l\}}
\varepsilon^{kJ_{\hat{k}}\cup \{l\}}_{J\cup \{l\}} L_l \Lb_k u_{kl}\omegab_J+\cdots
\end{equation*}
for $l\neq n$ and $J\not \owns n$,
 where here the $\cdots$ refer to terms
which are of the form $O(\rho) L_j
 + O(\rho)\Lb_j$, or are of 
order 0, or are terms orthogonal to 
 $\omegab_J$.
In the case $l=n$ we have
\begin{multline*}
\mdbar^{\ast} \mdbar ( \phi
u_{kn}
\omegab_{J_{\hat{k}}\cup\{n\}}
)
=\\
-(-1)^{|J|}
 \varepsilon^{kJ_{\hat{k}}}_{J}
\frac{1}{\sqrt{2}}\phi'(\rho)
\Lb_k  u_{kn}\omegab_J
-(-1)^{|J|}\varepsilon^{kJ_{\hat{k}}}_{J}
 \phi(\rho) L_n \Lb_k  u_{kn}
 \omegab_J +\cdots.
\end{multline*}

From these calculations, we see that
in a small neighborhood of a boundary point
$p\in \partial\Omega$,
 for which again we assume $p=0$,
the equation $2\square_{\phi}v=0$
corresponding to forms for which
$v_J=0$ if $n\in J$ can be written
\begin{align*}
- \Bigg(
\frac{\partial^2}{\partial\rho^2}+
\frac{1}{2} \sum_j \frac{\partial ^2}{\partial x_j^2}
+&
2 \frac{\partial ^2}{\partial x_{2n-1}^2}
+  \sum_{j,k=1}^{2n-1} l_{jk}
\frac{\partial ^2}
{\partial x_j\partial x_k}
\Bigg)v\\
& + \sqrt{2}S_{\phi} \left(\frac{\partial v}{\partial\rho}\right)
+ A_{\phi}  v+\rho\tau_{\phi}(v)
=0,
\end{align*}
where
\begin{equation*}
 S_{\phi} = S  
   - \sqrt{2}
  \phi'(\rho) 
  + O(\rho),
\end{equation*}
\begin{equation*}
 A_{\phi}
  = A +O(\rho) ,
\end{equation*}
and
\begin{equation*}
  \tau_{\phi}
   = \tau -2\phi(\rho) L_n\Lb_n
    + \cdots
\end{equation*}
where $S$, $A$, and $\tau$ are the operators 
 from Section \ref{dnosec}, and the 
$\cdots$ in the expression for $\tau_{\phi}$
 refer to second order terms which are
  $O(\rho)$, and are compositions with at least
one $L_k$ or $\Lb_k$ for $k\in 
 \{ 1,\ldots,n-1\}$ (this also holds true in
the case $n\in J$, although is not needed).

 We now examine the contributions from the
 $\phi$ function to the DNO.  
Using Lemma \ref{liglem}
 the $O(\rho)$ terms of the operator 
$S_{\phi}$ and $A_{\phi}$ 
 above lead to operators of order
 $-1$ and lower.  
From Theorem \ref{dno} we see the symbol
$s_0(x)$ for the DNO corresponding to 
 $2\square$ should be replaced with 
$s_0(x) - \sqrt{2} \phi'(0)$.

For the contributions from the $\tau_{\phi}$
 operator we expand
$\phi(\rho) = \phi'(0)\rho + O(\rho^2)$
and look at the terms
\begin{equation*}
 -2 \phi'(0)\rho \left(
 \frac{1}{2} \partial_{\rho}^2 
   +T^2
 \right)
\end{equation*}
coming from 
 $-2\phi(\rho) L_n\Lb_n$ in $\tau_{\phi}$.  
A term $\rho \partial_{\rho}^2 v$ can be written
using transforms, assuming 
 the support of $v$ is contained in a 
small coordinate patch around $0\in\partial
 \Omega$, as
\begin{equation*}
 \nonumber
 \rho \frac{1}{(2\pi)^{2n}}\int
 \left( - \eta^2 \widehat{v}(\xi,\eta)
 + \partial_{\rho} \widetilde{v}(\xi,0)
 + i \eta \widetilde{g}_b(\xi)
 \right) e^{ix\xi} e^{i\rho\eta} d\xi d\eta.
\end{equation*}
 Since
$\rho\cdot\delta(\rho)\equiv 0$, we have
\begin{equation*}
\rho 
\int \partial_{\rho}
 \widetilde{v}
(\xi,0)
e^{i\rho\eta} e^{ix\cdot \xi}
d\xi d\eta \equiv 0
\end{equation*}
in the term above,
 and 
\begin{equation}
 \label{extraTau}
\rho\partial_{\rho}^2 v =
  \rho \frac{1}{(2\pi)^{2n}}\int\left(
 - \eta^2\widehat{v}(\xi,\eta) 
  +i  \eta \widetilde{g}_b(\xi)\right)
e^{ix\xi} e^{i\rho\eta} d\xi d\eta.
\end{equation}

We examine first the term
 $\rho  \int \eta^2 \widehat{v} e^{ix\xi} e^{i\rho\eta} d\xi d\eta$, recalling that
$v$ can be written as 
 $v=\Theta^+ g $ modulo lower order terms:
\begin{align*}
 \rho  \frac{1}{(2\pi)^{2n}} \int\left(
- \eta^2\widehat{v}(\xi,\eta) 
\right)
&e^{ix\xi} e^{i\rho\eta} d\xi d\eta\\
=&-  
\rho \frac{i}{(2\pi)^{2n}}
\int
\frac{\eta^2}
{\eta+i|\Xi(x,\xi)|} \widetilde{g}_b(\xi)
e^{i\rho\eta} e^{i x\xi} d\eta d\xi
\\
=& -  
 \frac{1}{(2\pi)^{2n}}
\int
\frac{\eta^2}
{(\eta+i|\Xi(x,\xi)|)^2} \widetilde{g}_b(\xi)
e^{i\rho\eta} e^{i x\xi} d\eta d\xi
\\
&+ \frac{1}{(2\pi)^{2n}}
\int
\frac{2\eta}
{\eta+i|\Xi(x,\xi)|} \widetilde{g}_b(\xi)
e^{i\rho\eta} e^{i x\xi} d\eta d\xi,
\end{align*}
modulo lower order terms and smooth terms
 (of the form $R^{-\infty}$).
 In the calculation of the DNO, the above term
contributes 
\begin{equation}
 \label{contN2}
 2|\Xi(x,\xi)| \Gamma^{-1}_{int}
  \circ \phi'(0) \rho  
   \circ F.T.^{-1} \left(
   \eta^2 \widehat{v} \right)
\end{equation}
(see the calculation preceding Proposition
 \ref{L0Prop}).
We thus need
\begin{align*}
\phi'(0)
\Gamma^{-1}_{int}
\circ \rho  &
 F.T.^{-1} \left(
\eta^2 \widehat{v} \right)\\
=&  \frac{\phi'(0)}{(2\pi)^{2n}}
\int \frac{1}{\eta^2+\Xi^2(x,\xi)}
\frac{\eta^2}
{(\eta+i|\Xi(x,\xi)|)^2} \widetilde{g}_b(\xi)
e^{i\eta\rho} e^{i\xi x} d\eta d\xi
\\
&-\frac{\phi'(0)}{(2\pi)^{2n}}
\int \frac{1}{\eta^2+\Xi^2(x,\xi)}
\frac{2\eta}
{\eta+i|\Xi(x,\xi)|} \widetilde{g}_b(\xi)
e^{i\eta\rho} e^{i\xi x} d\eta d\xi
\\
=&  \frac{\phi'(0)}{(2\pi)^{2n}}
\int
\frac{1}{\eta-i|\Xi(x,\xi)|}
\frac{\eta^2}{\left(\eta+i|\Xi(x,\xi)|\right)^3}
 \widetilde{g}_b(\xi)
e^{i\eta\rho} e^{i\xi x} d\eta d\xi
\\
&-\frac{\phi'(0)}{(2\pi)^{2n}}
\int \frac{1}{\eta-i|\Xi(x,\xi)|}
\frac{2\eta}
{(\eta+i|\Xi(x,\xi)|)^2} \widetilde{g}_b(\xi)
e^{i\eta\rho} e^{i\xi x} d\eta d\xi,
\end{align*}
again, modulo lower order terms
 and smooth terms.
Integrating over $\eta$ and
setting $\rho=0$ yields
\begin{align*}
 \frac{\phi'(0)}{8}   
  \frac{1}{(2\pi)^{2n-1}}
  \int \frac{\widetilde{g}_b(\xi)}
{|\Xi(x,\xi)|}
&
e^{i\xi x}  d\xi - 
\frac{\phi'(0)}{2} 
 \frac{1}{(2\pi)^{2n-1}}     \int \frac{\widetilde{g}_b(\xi)}
{|\Xi(x,\xi)|}
e^{i\xi x}  d\xi,
\end{align*}
modulo $\Psi^{-2}_bg_b$ and smoothing terms.
 When setting the terms of 
order $-1$ in the $|\Xi(x,\xi)|$ factors
 equal as we did to show Proposition
 \ref{L0Prop}, we are led to the symbols
\begin{equation*}
  \frac{\phi'(0)}{4}   
  - 
 \phi'(0) 
 =   -\frac{3}{4}\phi'(0)   .
\end{equation*}
 for the contribution of \eqref{contN2}
in the DNO for the operator $2\square_{\phi}$.

We further need the contribution of the boundary
 term, $g_b$ in \eqref{extraTau} to the
DNO.  Similar to above, the contribution comes
 through
\begin{equation*}
 - 2|\Xi(x,\xi)| \phi'(0)
 \Gamma^{-1}_{int}
 \circ \rho  
 F.T.^{-1} \big(
 i\eta
  \widetilde{g}_b(\xi) \big)
\end{equation*}
 for which we have
\begin{align*}
 \phi'(0)
\Gamma^{-1}_{int}
\circ \rho & 
F.T.^{-1} \big(
i\eta
\widetilde{g}_b(\xi) \big) \Bigg|_{\rho=0}\\
 = &
 - \frac{\phi'(0)}{(2\pi)^{2n}}
  \int
\frac{2\eta^2}{(\eta^2+\Xi^2(x,\xi))^2}
 \widetilde{g}_b(\xi)
   e^{i\xi x} d\eta d\xi\\
 =&
  - \frac{\phi'(0)}{2}
  \frac{1}{(2\pi)^{2n-1}}
 \int 
\frac{\widetilde{g}_b(\xi)}{|\Xi(x,\xi)|}
 e^{i\xi x} d\eta d\xi,
\end{align*}
modulo lower order and smoothing terms.

As in the calculations of 
 Theorem \ref{dno} the 
$-2\phi'(0)\rho T^2$ terms lead to a term
 with
 symbol
\begin{equation*}
 - \frac{\phi'(0)}{2}
  \frac{\xi_{2n-1}^2}{\Xi^2(x,\xi)}
\end{equation*}
 in the DNO.

We note the $O(\rho)$ second order
 terms with at least one of 
$L_k$ or $\Lb_k$ with $k\in\{1,\ldots,n-1\}$
 lead to terms
$O\left(\frac{|\xi_L|}{|\Xi(x,\xi)|}
 \right)$.
Therefore, the contributions from
 the operator
$\tau_{\phi}$ in addition to those from 
 $\tau$ are given by
adding 
\begin{equation*}
-	\frac{3}{4}\phi'(0)
+\phi'(0)
	- \frac{\phi'(0)}{2}
	\frac{\xi_{2n-1}^2}{\Xi^2(x,\xi)} 
\end{equation*}
to the DNO for $2\square$.
 Note this term tends to 0 as 
$\xi_{2n-1} \rightarrow -\infty$.

 We thus have the following 
description of the DNO in a microlocal
neighborhood in the support of $\psi^-$:
\begin{prop}
\label{DNOcomp}
Modulo pseudodifferential operators of order $-1$, the symbol for
$N^{\phi,-}$ is given by
\begin{align*}
\sigma(N^{\phi,-}) (x,\xi)
	=& \sigma(N^-) (x,\xi)
	- \phi'(0) + O\left(\frac{|\xi_L|}{|\Xi(x,\xi)|}
	\right) .
	\end{align*}
\end{prop}

 Returning to the boundary conditions,
we see how the additional terms 
 from the DNO coming from the added
$\phi(\rho)$ function affect the boundary 
 equations \eqref{modBC}.
 The first condition,
$u\rfloor \mdbar\rho$ remains the same, and
 is equivalent to 
$u_{J} = 0$ if $n\in J$.

We recall the second condition written as in
 \eqref{bndryExpnd}:
\begin{equation*}
\left(\frac{1}{\sqrt{2}}  N^{\phi,-}
-iT^0 \right) u_{bJ}
+ \left( \frac{1}{\sqrt{2}} \phi'(0)
+(-1)^{|J|}c_{Jn}^J \right) u_{bJ}
= R\circ \Psi^{-1}(f).
\end{equation*}
From Proposition \ref{DNOcomp} we can
 write
$\sigma (N^{\phi,-}) = \sigma (N^-)
  -\phi'(0) + 
O\left(\frac{|\xi_L|}{|\Xi(x,\xi)|}
\right),$ modulo lower order
 symbols.  In particular, the
 $\phi'(0)$ term in the boundary equation
cancels with that coming from the DNO.  We can
 state the
\begin{thrm}
	Let $\square_{\phi}
	= \mdbar \mdbar^{\ast}
	+ \mdbar^{\ast}\mdbar\circ (1+\phi)$.
Let $\phi=\phi(\rho)$ be a smooth
 function which depends only on the
defining function, with the
 property $\phi(\rho) = O(\rho)$.
The condition $\mdbar\circ (1+\phi) u \in
 \mbox{dom}(\mdbar^{\ast})$, equivalent to
$\mdbar  \big((1+\phi)u
\big) \rfloor \mdbar \rho = 0$ on 
 $\partial\Omega$, has the form 
\begin{equation*}
\left( \frac{1}{\sqrt{2}} N^-_1
-i T^0\right) u_{b,J} +
\Upsilon^0_J u_{b}
=  R\circ \Psi^{-1}f
\end{equation*} 
 as in \eqref{nMinusTUpOrig} of Proposition
 \ref{propNMinusT}, with
$\Upsilon^0_J$ sharing the same properties
 as those of \eqref{upDef}, \eqref{upProp1}
  and \eqref{upProp2}.
\end{thrm}

\end{document}